\documentclass[a4paper,10pt]{amsart}
\usepackage[utf8]{inputenc}
\usepackage[T1]{fontenc}
\usepackage{amsmath}
\usepackage{amsfonts}
\usepackage{amssymb}
\usepackage{amsthm}
\usepackage{mathrsfs}

\usepackage{graphicx}

\usepackage[all]{xy}
\usepackage{tikz}
\usetikzlibrary{matrix,arrows,decorations.pathmorphing}
\usepackage{tikz-cd}

\usepackage[mathcal]{euscript}

\usepackage{stmaryrd} 

\DeclareMathOperator{\Hom}{\mathsf{Hom}}

\DeclareMathOperator{\Ext}{\mathsf{Ext}}

\DeclareMathOperator{\Ann}{\mathsf{Ann}}
\DeclareMathOperator{\tr}{\mathsf{tr}}
\DeclareMathOperator{\Ker}{Ker}
\DeclareMathOperator{\Coker}{Coker}
\DeclareMathOperator{\im}{Im}

\DeclareMathOperator{\Hocolim}{Hocolim}

\DeclareMathOperator{\Gen}{\mathsf{Gen}}
\DeclareMathOperator{\Cogen}{\mathsf{Cogen}}
\DeclareMathOperator{\Spec}{\mathsf{Spec}}

\DeclareMathOperator{\Supp}{\mathsf{Supp}}
\DeclareMathOperator{\supp}{\mathsf{supp}}

\newcommand{\Mod}{\mathsf{Mod}}
\renewcommand{\mod}{\mathsf{mod}}

\newcommand{\D}{{\sf{D}}}

\newcommand{\fp}{\mathsf{fp}}
\newcommand{\co}{\mathsf{c}}
\newcommand{\spcl}{\mbox{\larger\larger$\vee$}}
\newcommand{\gncl}{\mbox{\larger\larger$\wedge$}}
\newcommand{\real}{\mathsf{real}}
\newcommand{\realb}{\mathsf{real}^b}
\renewcommand{\paragraph}[1]{\par\textbf{#1}}

\newtheorem{thm}{Theorem}
\newtheorem{lemma}[thm]{Lemma}
\newtheorem{prop}[thm]{Proposition}
\newtheorem{cor}[thm]{Corollary}
\theoremstyle{definition}
\newtheorem{dfn}[thm]{Definition}
\newtheorem{rmk}[thm]{Remark}
\newtheorem{ex}[thm]{Example}
\newtheorem{ntn}[thm]{Notation}
\usepackage{fullpage}
\numberwithin{equation}{section}
\numberwithin{thm}{section}

\newcommand{\Acal}{\ensuremath{\mathcal{A}}}
\newcommand{\Bcal}{\ensuremath{\mathcal{B}}}
\newcommand{\Ccal}{\ensuremath{\mathcal{C}}}
\newcommand{\Dcal}{\ensuremath{\mathcal{D}}}
\newcommand{\Dbb}{\ensuremath{\mathbb{D}}}
\newcommand{\Fcal}{\ensuremath{\mathcal{F}}}
\newcommand{\Gcal}{\ensuremath{\mathcal{G}}}
\newcommand{\Hcal}{\ensuremath{\mathcal{H}}}

\newcommand{\Lcal}{\ensuremath{\mathcal{L}}}

\newcommand{\Pcal}{\ensuremath{\mathcal{P}}}

\newcommand{\Scal}{\ensuremath{\mathcal{S}}}
\newcommand{\Tcal}{\ensuremath{\mathcal{T}}}
\newcommand{\Tbb}{\ensuremath{\mathbb{T}}}
\newcommand{\Ucal}{\ensuremath{\mathcal{U}}}
\newcommand{\Vcal}{\ensuremath{\mathcal{V}}}
\newcommand{\Wcal}{\ensuremath{\mathcal{W}}}
\newcommand{\Xcal}{\ensuremath{\mathcal{X}}}
\newcommand{\Ycal}{\ensuremath{\mathcal{Y}}}
\newcommand{\Zcal}{\ensuremath{\mathcal{Z}}}
\newcommand{\Zbb}{\ensuremath{\mathbb{Z}}}

\begin{document}
\title{Hearts for commutative noetherian rings: \\ torsion pairs and derived equivalences}
\author{Sergio Pavon, Jorge Vit\'oria}
\address{Sergio Pavon, Dipartimento di Matematica "Tullio Levi-Civita", Universit\`a degli Studi di Padova, Via Trieste, 63, 35121 Padova, Italy} 
\email{sergio.pavon@math.unipd.it}
\address{Jorge Vit\'oria, Dipartimento di Matematica e Informatica, Universit\'a degli Studi di Cagliari, Palazzo delle Scienze, Via Ospedale, 72, 09124 Cagliari, Italy}
\email{jorge.vitoria@unica.it}
\urladdr{https://sites.google.com/view/jorgevitoria/}
\thanks{The authors are grateful to Frederik Marks and Alexandra Zvonareva for sharing with them
results of their work in progress \cite{MZ}, and to Lidia Angeleri H\"ugel and Ryo Takahashi for
Lemma \ref{cosygyzy}. The authors acknowledge financial support from the Fondazione Cariverona under the programme \textit{Ricerca Scientifica di Eccellenza 2018}, project: \textit{Reducing complexity in algebra, logic, combinatorics - REDCOM}. The first-named author was also partially supported by the Department of Mathematics Tullio Levi-Civita of the University of Padova (Project BIRD163492/16 \textit{Categorical homological methods in the study of algebraic structures} and the Research programme DOR1828909 \textit{Anelli e categorie di moduli.) }}
\maketitle
\begin{center} {\textit{Dedicated to Lidia Angeleri H\"ugel on the occasion of her 60th birthday}} \end{center}
\begin{abstract}
Over a commutative noetherian ring $R$, the prime spectrum controls, via the assignment of support,
the structure of both $\Mod(R)$ and $\D(R)$. We show that, just like in $\Mod(R)$, the assignment of
support classifies hereditary torsion pairs in the heart of any nondegenerate compactly generated
$t$-structure of $\D(R)$. Moreover, we investigate whether these $t$-structures induce derived
equivalences, obtaining a new source of Grothendieck categories which are derived equivalent to
$\Mod(R)$.
\end{abstract}


\section{Introduction}
Two common techniques to study the category of modules over a ring and its derived counterpart are:
\textit{decompose into simpler parts} and \textit{replace by suitably equivalent category}.
Concerning the former, the notion of a torsion pair is central. Modelled on the concept of torsion
for abelian groups, these are pairs of subcategories that generate, by extensions, the whole
category, while simultaneously being orthogonal under the $\Hom$-pairing
(\cite{D}). This idea extends beyond abelian categories to triangulated categories, giving rise in
particular to the important concepts of $t$-structures and recollements (\cite{BBD}).  Torsion pairs
are closely linked to categorical localisations, another useful tool in representation theory.
Concerning categorical equivalences, we have at our disposal the theory of Morita, studying equivalences
of module categories, and its derived version, studying triangle equivalences between derived
categories of modules (\cite{R}). In recent years, derived Morita theory has been extended to
include equivalences between derived categories of Grothendieck categories (\cite{PV},\cite{V}).
These include not only categories of modules over a ring but also, for example, categories of
quasi-coherent sheaves over a scheme.

In this paper we look at Grothendieck categories naturally occurring in the derived category of a
commutative noetherian ring as hearts of $t$-structures. We approach the problem of classifying their
localising subcategories (or, equivalently, hereditary torsion classes) and the problem of
determining whether they are derived equivalent to the module category.  In the noncommutative
setting, these questions are, in general, quite difficult even in the case of a finite-dimensional
algebra. In fact, a complete classification of localising
subcategories is not even available for the module category of the Kronecker algebra (\cite{KS}). Understandably, the same
classification problem in the derived category is even more cumbersome. With regards to derived equivalences, there is often a triangle functor linking the bounded derived category of a given heart and that of the ring, and there are criteria to check whether such a functor is an equivalence (see \cite{BBD,CHZ}). However, these criteria are in general not easy to check. 

A remarkable feature of the representation theory of a commutative noetherian ring $R$ is that much
of the structure of the modules category or of its derived counterpart is controlled by the Zariski
spectrum of the ring. In fact, some of the problems mentioned in the above paragraphs have been
elegantly solved in the module category and in the derived category. On one hand, Gabriel classified
localising subcategories of $\mathsf{Mod}(R)$ in terms of specialisation-closed subsets of the
spectrum, and Neeman classified localising subcategories of $\D(R)$ in terms of arbitrary subsets of
$\Spec(R)$. Even compactly generated $t$-structures are completely classified in such derived
categories (\cite{AJS}). On the other hand, it is well-known that if two commutative noetherian
rings are derived equivalent, then they are isomorphic (see, for example, \cite[Proposition
5.3]{PS3} for a stronger statement). Our aim in this paper is to show that the prime spectrum also
controls \textit{localising subcategories} of Grothendieck hearts in the derived category of $R$ and,
moreover, that these hearts are often \textit{derived equivalent} to $\mathsf{Mod}(R)$. In fact, we prove the
following.
\pagebreak

\noindent \textbf{Theorem.} \textit{Let $R$ be a commutative noetherian ring and let $\mathbb{T}$
denote a nondegenerate compactly generated $t$-structure with heart $\Hcal_\Tbb$. Then the following
holds:
\begin{enumerate}
\item (Proposition \ref{some properties} and Theorem \ref{classification of tp in heart}(1)(a)) The
hereditary torsion classes of $\Hcal_\Tbb$ are completely determined by their support in
$\Spec(R)$. If $\Tbb$ is intermediate (Definition
\ref{dfn:t-structure-properties}), every specialisation-closed subset of $\Spec(R)$ is the support of a
hereditary torsion class in $\Hcal_\Tbb$.
\item If $\mathbb{T}$ is intermediate and restricts to $\D^b(\mathsf{mod}(R))$, then 
\begin{enumerate}
\item (Theorem \ref{der eq restrictable}) $\mathbb{T}$ induces a derived equivalence between $\Hcal_\Tbb$ and $\mathsf{Mod}(R)$;
\item (Theorem \ref{classification of tp in heart} and Corollary \ref{tp finite type restrictable})
There is a bijection between hereditary torsion pairs of finite type and specialisation-closed
subsets of $\Spec(R)$.
\end{enumerate}
In particular, the above two assertions hold for the HRS-tilt of the standard $t$-structure at a
hereditary torsion pair in $\Mod(R)$.
\end{enumerate}}

For some hearts, we are able to give a complete description of the supports on hereditary torsion
pairs (see Subsection 4.2). While the results obtained about torsion pairs rely on the
well-developed theory of support, the results on derived equivalences instead rely on the study of
HRS-tilts. Happel, Reiten and Smal{\o} developed in \cite{HRS} a way to create a new $t$-structure
from an old one, provided we are given a torsion pair in the old heart. The properties of this new
$t$-structure depend on the properties of the given torsion pair, and therefore one may say that
studying HRS-tilts often can be reduced to studying the associated torsion pairs. However, HRS-tilts
turn out to be an elementary operation that, when iterated, allows us obtain a large class of
$t$-structures (see, for example, \cite{FMT,MPT}). Moreover, HRS-tilts turn out to play an important
role in understanding Bridgeland's stability condition manifold (see, for example, \cite{Bri},
\cite{BPP} and \cite{PZ}). In \cite{CHZ}, necessary and sufficient conditions for an HRS-tilt to
induce a derived equivalence were studied, and we review this theorem in Section 5. We use this to
prove the following result, that becomes a fundamental tool in our application to commutative rings
in the theorem above.
\vspace{0.2cm}

\noindent \textbf{Theorem.} \textit{(Theorems \ref{criterion Grothendieck} and
\ref{thm:her-criterion} and Corollary \ref{der eq comm ring}) Let $\Gcal$ be a Grothendieck abelian category with generator $G$, and let
$\mathbf{t}=(\Tcal,\Fcal)$ be a hereditary torsion pair in $\Gcal$ with torsion radical $t\colon
\Gcal\to \Tcal$. Then $\mathbf{t}$ induces an equivalence of bounded derived categories if and only
if $G/\mathsf{tr}_{G/t(G)}(G)$ lies in $\Tcal$.  Moreover, if $\Gcal=\mathsf{Mod}(R)$, then
$\mathbf{t}$ induces an equivalence of bounded derived categories if and only if
$R/\mathsf{Ann}({}_Rt(R))$ lies in $\Tcal$. As a consequence, if $R$ is commutative and noetherian, then every hereditary torsion pair induces a (bounded) derived equivalence.} 
\vspace{0.1cm}

The results concerning derived equivalences between Grothendieck hearts and $\mathsf{Mod}(R)$, for
$R$ commutative noetherian, have implications on the level of silting theory. Namely, it allows us
to show that any bounded cosilting complex in $\D(R)$ whose associated $t$-structure restricts to
$\D^b(\mathsf{mod}(R))$ must in fact be a cotilting complex (Corollary \ref{bdd cosilting is
cotilting}). In particular, every two-term cosilting complex is, in this setting, necessarily
cotilting (Corollary \ref{2-term cosilting is cotilting}). These results may easily lead to the
expectation that every bounded cosilting is cotilting. This is, however, not true, as shown in
Example \ref{ex:Z-secondtilt}. 
\vspace{0.2cm}

\noindent\textbf{Structure of the paper.} In Section~2 we recall some definitions and known results about torsion pairs and $t$-structures.
In Section~3 we turn to the definition of support over a commutative noetherian ring and we collect
some of the known classification results for various types of subcategories of $\Mod(R)$ and
$\D(R)$.  In Section~4 we prove that hereditary torsion pairs in the heart of a nondegenerate
compactly generated $t$-structure of $\D(R)$ are determined by their supports  (Proposition
\ref{some properties}). We investigate the sets arising as support of hereditary torsion classes in
such hearts. In Subsection 4.2 we are able to describe these subsets of $\Spec(R)$ for some
specific hearts. In Section~5 we temporarily leave the commutative noetherian setting to address the
question of when a hereditary torsion pair in a Grothendieck category $\mathcal G$ gives rise to a
derived equivalent category via HRS-tilting. We then specialise the results to module categories. In
Section~6 we return to the commutative noetherian case. We observe that the intermediate compactly
generated $t$-structures of $\D(R)$ can be obtained from the standard one via a finite chain of
HRS-tilting operations (Proposition \ref{iterated HRS}), with respect to hereditary torsion pairs of finite type at
each step. Combining the results of Sections~4~and~5 we show that if the intermediate compactly
generated $t$-structure at the end of the chain is restrictable, then across each HRS-tilting step
of the chain there is a derived equivalence (Theorem \ref{der eq restrictable}).  This provides a
class of $t$-structures in $\D(R)$ whose heart is derived equivalent to $\Mod(R)$. In particular, we
conclude that a bounded cosilting objects whose associated $t$-structure is restrictable must be
cotilting (Corollary \ref{bdd cosilting is cotilting}).
\vspace{0.2cm}

\noindent\textbf{Notation and conventions.} All subcategories considered in this paper are strict and full. 
Given a class $\Scal$ of objects of a category $\Xcal$, we denote by $\Gen(\Scal)$ (respectively, $\Cogen(\Scal)$)
the subcategory of $\Xcal$ formed by the epimorphic images of existing coproducts (respectively, subobjects
of existing products) of objects in $\Scal$.
If $\Xcal$ is preadditive (e.g. abelian or triangulated) we write
\[\Scal^{\perp}:=\{Z\in\Xcal\colon \Hom_\Xcal(S,Z)=0 \ \forall S\in\Scal\}\qquad
  {}^{\perp}\Scal:=\{Z\in\Xcal\colon \Hom_\Xcal(Z,S)=0 \ \forall S\in\Scal\} \]
If $\Xcal$ is abelian and $I\subseteq\mathbb{N}_0$ is a set of naturals,
we write
$$\Scal^{\perp_I}:=\{Z\in\Xcal\colon \Ext_\Xcal^k(S,Z)=0, \ \forall k\in I, \forall S\in\Scal\}.$$
If $\Xcal$ is a
triangulated category and $J\subseteq \mathbb{Z}$, we write
\[
\Scal^{\perp_J} := \{ Z\in\Xcal \colon \Hom_\Xcal(S,Z[k])=0,\ \forall k\in J,S\in\Scal \}, \ 
{}^{\perp_J}\Scal := \{ Z\in\Xcal\colon \Hom_\Xcal(Z,S[k])=0, \ \forall k\in J,S\in\Scal \}
\]
Often we replace the subsets $I$ and $J$ above by expressions of the form $>0$, $\geq 0$, $\leq 0$,
$<0$, or simply a list of integers, with the obvious meaning. If $\Xcal$ is
abelian (or triangulated) and $\Ycal$ and $\Zcal$ are subcategories of $\Xcal$,
we denote by $\Ycal\ast\Zcal$ the subcategory of $\Xcal$ formed by the objects
$X$ for which there are $Y$ in $\Ycal$, $Z$ in $\Zcal$ and a short exact
sequence (respectively, a triangle)
\[0\longrightarrow Y\longrightarrow X\longrightarrow Z\longrightarrow 0\]
\[(\text{respectively,}\  Y\longrightarrow X\longrightarrow Z\longrightarrow X[1]).\]
If $R$ is a ring, all
modules considered are right modules. The category of $R$-modules is
denoted by $\Mod(R)$, and its subcategory of finitely presented modules by
$\mod(R)$. We denote by $\D(R)$ the unbounded derived category of $\Mod(R)$ and by $\D^+(R)$ (respectively, $\D^b(R)$) 
the bounded below (respectively, bounded) counterpart. The bounded derived category of
finitely presented $R$-modules is denoted by $\D^b(\mod(R))$.

\section{Preliminaries: torsion pairs and $t$-structures}

Most statements in this section about abelian or triangulated categories hold under more general
assumptions than those presented. For simplicity, we restrict ourselves to the settings of
Grothendieck abelian categories and derived categories of module categories. The generality under
which each result holds can be extracted by the references provided. An abelian
category $\Gcal$ is said to be \textbf{Grothendieck} if it admits arbitrary (set-indexed)
coproducts, direct limits are exact in $\Gcal$, and $\Gcal$ has a generator.
In a triangulated category with arbitrary (set-indexed) coproducts, an object $X$ in $\Dcal$ is said
to be \textbf{compact} if $\Hom_\Dcal(X,-)$ commutes with coproducts, and $\Dcal$ is said to
be \textbf{compactly generated} if the subcategory of compact objects, denoted by $\Dcal^c$, is
skeletally small and $(\Dcal^c)^\perp=0$.
For a ring $R$, $\D(R)$ is compactly generated and $\D(R)^c$ is the subcategory of bounded complexes of finitely generated projective $R$-modules. 

\subsection{Torsion pairs}
Torsion pairs are useful (orthogonal) decompositions of abelian categories.
\begin{dfn}
A pair $(\mathcal T,\mathcal F)$ of subcategories of an abelian category $\mathcal A$ is said to be
a \textbf{torsion pair} if
\begin{enumerate}
\item $\Hom_\mathcal A(T,F)=0$ for any $T$ in $\mathcal T$ and any $F$ in $\mathcal F$.
\item $\mathcal A=\mathcal T \ast \mathcal F$.
\end{enumerate}
If $\mathbf{t}:=(\mathcal T,\mathcal F)$ is a torsion pair in $\mathcal A$, $\mathcal T$ is called its
\textbf{torsion class} and $\mathcal F$ its \textbf{torsionfree class}. The pair $\mathbf{t}$ is said to be
\textbf{hereditary} if $\Tcal$ is closed under
subobjects, and \textbf{cohereditary} if $\Fcal$ is closed under quotient objects. A subcategory
$\Vcal$ of $\Acal$ is called a \textbf{torsion torsionfree class} (TTF class, for short) if
there are torsion pairs $(\Ucal,\Vcal)$ and $(\Vcal,\Wcal)$ in $\Acal$, i.e. if $\Vcal$ is both a
torsion class and a torsionfree class in $\Acal$.
\end{dfn}

A subcategory $\Xcal$ of a Grothendieck category $\Acal$ is a
torsion class if and only if it is closed under coproducts, quotient objects and extensions, and it
is a torsionfree class if and only if it is closed under products, subobjects and extensions (\cite{D}). Since the coproduct of a family of objects in $\Acal$ is a subobject of the product of
that same family, $\Xcal$ is a TTF class if and only if it is a
cohereditary torsionfree class. 

\begin{ex}
Given a Grothendieck abelian category $\mathcal{G}$ and a set of objects $\Xcal$ of $\mathcal{G}$,
it follows from the description above of torsion and torsionfree classes that the pairs
$(\Tcal_\Xcal,\Xcal^\perp)$ and $({}^\perp \Xcal,\Fcal_\Xcal)$ are torsion pairs. The first is said
to be \textbf{generated by $\Xcal$} while the second is said to be \textbf{cogenerated by $\Xcal$}.
\end{ex}

Since a torsion class is closed under coproducts and quotients, it is always closed for direct
limits; this in general is not the case for torsionfree classes.
\begin{dfn}
A torsion pair in a Grothendieck abelian category is said to be \textbf{of finite type} if its
torsionfree class is closed under direct limits.
\end{dfn}

An object $X$ in a cocomplete abelian category $\Acal$ is said to be \textbf{finitely presented} if
$\Hom_\Acal(X,-)$ commutes with direct limits. The subcategory of finitely presented objects
of $\Acal$ will be denoted by $\fp(\Acal)$.  Recall from \cite{CB} that a Grothendieck
category $\Gcal$ is \textbf{locally finitely presented} provided that the subcategory $\fp(\Gcal)$
is skeletally small and every object of $\Gcal$ can be
expressed as the direct limit of a system of finitely presented objects.
Moreover, $\Gcal$ is \textbf{locally coherent} if it is locally finitely presented  and the
subcategory $\fp(\Gcal)$ is an abelian subcategory. For any ring $R$, $\Mod(R)$ is always
locally finitely presented; the ring is said to be \textbf{coherent} precisely when
$\Mod(R)$ is locally coherent. In particular, the category of modules over a noetherian
ring is locally coherent (in the example of the noetherian ring, $\fp(\Mod(R))=\mod(R)$ is the
category of finitely generated modules).

\begin{dfn}
If $\Gcal$ is locally coherent, a torsion pair $\mathbf{t}=(\Tcal,\Fcal)$ in $\Gcal$ is said to be
\textbf{restrictable} if $\mathbf{t}\cap\fp(\Gcal):=(\Tcal\cap\fp(\Gcal),\Fcal\cap\fp(\Gcal))$ is a
torsion pair in $\fp(\Gcal)$. 
\end{dfn}

The following well-known result relates restrictable torsion pairs to those of finite type.

\begin{lemma}\label{CB+H}
Let $\Gcal$ be a locally coherent Grothendieck category, and $\mathbf{t}=(\Tcal,\Fcal)$ a torsion
pair in $\Gcal$.
\begin{enumerate}
\item If $\mathbf{t}$ is generated by a set of finitely presented objects
	$\Scal\subseteq\fp(\Gcal)$, then $\mathbf{t}$ is of finite type. 
\item\cite[Lemma 2.3]{Kr} If $\mathbf{t}$ is a hereditary torsion pair of finite type, then
	$\Tcal=\varinjlim(\Tcal\cap\fp(\Gcal))$.
\item\cite[Lemma 4.4]{CB} If $\mathbf{t}$ is restrictable, then it is of finite type if and only if
	\[\mathbf{t}=\varinjlim(\mathbf{t}\cap\fp(\Gcal)):=
		(\varinjlim(\Tcal\cap\fp(\Gcal)),\varinjlim(\Fcal\cap\fp(\Gcal))).\]
\end{enumerate}
\end{lemma}
\begin{proof}
We comment on statement (3). If $\Fcal=\varinjlim(\Fcal\cap\fp(\Gcal))$, $\mathbf{t}$ is clearly of finite type.
For the converse, by \cite[\S 4.4]{CB}, since $\mathbf{t}\cap\fp(\Gcal)$ is a torsion pair in
$\fp(\Gcal)$, then  $\varinjlim(\mathbf{t}\cap\fp(\Gcal))$ is a torsion pair in $\Gcal$. Now,
if $\mathbf{t}$ is of finite type, we have both $\varinjlim(\Tcal\cap\fp(\Gcal))\subseteq\Tcal$ and
$\varinjlim(\Fcal\cap\fp(\Gcal))\subseteq\Fcal$; for the converse inclusion, it suffices to
consider the torsion decomposition sequences of objects of $\Tcal$ and $\Fcal$ with respect to
$\varinjlim(\mathbf{t}\cap\fp(\Gcal))$.
\end{proof}

\begin{rmk}
We will prove in Lemma \ref{cor:finite-types} the converse of item (1) when $\Gcal$ is the heart
of a compactly generated $t$-structure in the derived category of a commutative noetherian ring.
\end{rmk}

\begin{rmk}\label{rmk:here-finitetype}
For a right noetherian ring $R$, item (1) of Lemma \ref{CB+H} guarantees that every hereditary
torsion pair is of finite type. Indeed, recall that over any ring a module is the sum of its
finitely generated submodules: hence a hereditary torsion pair is generated by the class of finitely
generated torsion modules. If $R$ is right noetherian, these are automatically also finitely presented. Note, furthermore, that any torsion pair over a right noetherian ring is restrictable.
\end{rmk}

The following lemma, which will be useful later on, states that a torsionfree class in a
Grothendieck category admits a generator, in the sense below. 

\begin{lemma}\label{torsionfree generator}
Let $\Gcal$ be a Grothendieck category with generator $G$, and let
$\mathbf{t}=(\Tcal,\Fcal)$ be a torsion pair in $\Gcal$ with torsion radical $t\colon \Gcal\to
\Tcal$. Then every object in $\Fcal$ is a quotient of a coproduct of copies of $G/t(G)$, i.e.
$\Fcal\subseteq \Gen(G/t(G))$.
\end{lemma}
\begin{proof}
Let $X$ be an object of $\Fcal$ and let $f\colon G^{(I)}\to X$ be an epimorphism. Since both $\Tcal$ and
$\Fcal$ are closed under coproducts, it is easy to see that $t(G^{(I)})\cong t(G)^{(I)}$ and that
$G^{(I)}/t(G^{(I)})\cong (G/t(G))^{(I)}$. Hence, since $\Hom_\Gcal(t(G^{(I)}),X)=0$, the morphism
$f$ factors through an epimorphism $\overline{f}\colon (G/t(G))^{(I)}\longrightarrow X$.
\end{proof}

\subsection{$t$-structures}
The role of $t$-structures in triangulated categories, as first defined in \cite{BBD}, is analogous to
that of torsion pairs in abelian categories. Recall that, as set up in the
Introduction, given subcategories $\Ucal$ and $\Vcal$ in a triangulated category
$\Tcal$, $\Ucal\ast \Vcal$ stands for the subcategory given by extensions of
$\Ucal$ by $\Vcal$.

\begin{dfn}\label{t-def}
A pair of subcategories $\Tbb:=(\Ucal,\Vcal)$ in a triangulated category $\Dcal$ is a
\textbf{$t$-structure} if
\begin{enumerate}
\item $\Hom_\Dcal(\Ucal,\Vcal)=0$;
\item $\Ucal[1]\subseteq \Ucal$;
\item $\Ucal\ast\Vcal=\Dcal$.
\end{enumerate}
In that case, we call $\Ucal$ the \textbf{aisle}, $\Vcal$ the \textbf{coaisle} and
$\Hcal_\Tbb=\Vcal[1]\cap \Ucal$ the \textbf{heart} of $\Tbb$. The $t$-structure
$\Tbb$ is said to be \textbf{nondegenerate} if
$\cap_{n\in\mathbb{Z}}\Ucal[n]=0=\cap_{n\in\mathbb{Z}}\Vcal[n]$ and \textbf{bounded} if
$\cup_{n\in\mathbb{Z}}\Ucal[n]=\Dcal=\cup_{n\in\mathbb{Z}}\Vcal[n]$.
\end{dfn}

It is well-known from \cite{BBD} that the heart of a $t$-structure $\Tbb$ in $\Dcal$ is an
abelian category and that there is an associated cohomological functor $H^0_\Tbb\colon
\Dcal\to \Hcal_\Tbb$ which restricts to the identity functor in $\Hcal_\Tbb$.

\begin{ex}
For a ring $R$, let $H^0\colon \D(R)\longrightarrow \mathsf{Mod}(R)$ denote the standard cohomology
functor of the derived category of $R$. This
functor arises from a $t$-structure, also called \textbf{standard}. This is the pair
$\Dbb=(\Dbb^{\leq 0},\Dbb^{\geq 1})$ where each of the subcategories of $\D(R)$ in the pair
are made of the complexes whose non-zero cohomologies are concentrated in the degrees indicated in
superscript. Note that shifts of $t$-structures are also $t$-structures; we write
$\Dbb^{\leq k}=\Dbb^{\leq 0}[-k]$ and $\Dbb^{\geq k}=\Dbb^{\geq 1}[1-k]$ for shifts of the
standard aisle and coaisle. Note that the standard $t$-structure in $\D(R)$ is nondegenerate but not
bounded, while its restriction to the subcategory of bounded complexes $\D^b(R)$ is both
nondegenerate and bounded. 
\end{ex}

In the derived category of a ring $R$ we may consider a useful notion of \textbf{directed homotopy
colimit}: this is the derived functor of the direct limit functor. Since direct limits are exact in
any Grothendieck category (and, hence, in $\mathsf{Mod}(R)$), the directed homotopy colimit of a
directed system of complexes of $R$-modules is the object of $\D(R)$ obtained by applying the direct limit functor of $\mathsf{Mod}(R)$ componentwise. We
now recall some properties of $t$-structures. 

\begin{dfn}\label{dfn:t-structure-properties}
For a ring $R$, a $t$-structure $\Tbb=(\Ucal,\Vcal)$ in $\D(R)$ is said to be
\begin{itemize}
\item \textbf{intermediate} if there are integers $a<b$ such that $\Dbb^{\leq
a}\subseteq\Ucal\subseteq \Dbb^{\leq b}$;
\item \textbf{smashing} if $\Vcal$ is closed under coproducts;
\item \textbf{homotopically smashing} if $\Vcal$ is closed under directed homotopy colimits;
\item \textbf{compactly generated} if there is a set of compact objects $\Scal$ in $\Ucal$ such that
	$\Vcal=\Scal^\bot$.
\item \textbf{cosilting} if $\Tbb=({}^{\bot_{\leq0}}C,{}^{\bot_{>0}}C)$ for
	some object $C$ of $\D(R)$ (then called a \textbf{cosilting} object);
\item \textbf{cotilting} if it is cosilting and $C^I$ lies in ${}^{\bot_{\neq0}}C$ for all sets $I$ (in this case
	$C$ is said to be \textbf{cotilting}).
\end{itemize}
If $R$ is coherent, we say that $\Tbb$ is \textbf{restrictable} if the pair $(\Ucal\cap
\D^b(\mathsf{mod}(R)),\Vcal\cap\D^b(\mathsf{mod}(R)))$ is a $t$-structure of the triangulated
subcategory $\D^b(\mathsf{mod}(R))$ of $\D(R)$.
\end{dfn}

Compactly generated triangulated categories $\Dcal$ are often studied through the help of the
category  of additive functors $(\Dcal^c)^{op}\to \Mod(\mathbb{Z})$.  This category, usually denoted
by $\mathsf{Mod}(\Dcal^c)$, is a locally coherent Grothendieck category and the functor $\mathbf{y}$
sending an $X$ object of $\Dcal$ to the functor $\mathbf{y}X:=\mathsf{Hom}_\Dcal(-,X)_{|\Dcal^c}$ is
a cohomological functor. 
Properties of the Grothendieck category $\mathsf{Mod}(\Dcal^c)$ are reflected on properties of
$\Dcal$ via $\mathbf{y}$, and this allows us to distinguish objects of $\Dcal$ by the way they relate to
the subcategory of compact objects (see \cite{Bel} and \cite{KrauseInv} for more details). An
important class of objects of $\Dcal$ is given by those $X$ whose corresponding functor
$\mathbf{y}X$ is injective in the category $\mathsf{Mod}(\Dcal^c)$ --- these objects are called
\textbf{pure-injective}. It turns out that, in many contexts, cosilting (and cotilting) objects are
automatically pure-injective (see \cite{Ba0,Stov} and \cite{MV}). It is not known whether every
cosilting object in a compactly generated triangulated category is necessarily pure-injective. 

The following theorems relate properties of the $t$-structures with properties of their hearts. The
first one is a particular case of \cite[Theorem 4.6]{L}, adding to a sequence of previous results in
\cite{AHMV} and \cite{SSV}.

\begin{thm}\label{thm:$t$-structure-types}
The following are equivalent for a nondegenerate $t$-structure $\Tbb$ of $\D(R)$, for a ring $R$.
\begin{enumerate}
	\item $\Tbb$ is homotopically smashing;
	\item $\Tbb$ is smashing and its heart is a Grothendieck category;
	\item $\Tbb$ is cosilting, for a pure-injective cosilting object.
\end{enumerate}
Moreover, a compactly generated $t$-structure has the above properties.
\end{thm}

For intermediate $t$-structures with the properties of the theorem above, the
question of whether or not they are restrictable
boils down to a property of the heart, as follows.

\begin{thm}\cite{MZ}\label{MZ}
Let $R$ be a noetherian ring and let $\Tbb$ be an intermediate smashing $t$-structure in $\D(R)$ with a
Grothendieck heart $\Hcal_\Tbb$. Then $\Tbb$ is restrictable if and only if $\Hcal_\Tbb$ is locally
coherent and $\mathsf{fp}(\Hcal_\mathbb{T})=\Hcal_\mathbb{T}\cap \D^b(\mathsf{mod}(R))$.
\end{thm}

We conclude this subsection by recalling an important and much studied source of $t$-structures, which
will be central to our paper.

\begin{thm}[\cite{HRS}]
Let $\Dcal$ be a triangulated category and $\Tbb$ a $t$-structure in $\Dcal$, with heart
$\Hcal_\mathbb{T}$. Given a torsion
pair $\mathbf{t}=(\Tcal, \Fcal)$ in $\Hcal_\Tbb$, there is a $t$-structure
$\Tbb_\mathbf{t}:=(\Tbb_\mathbf{t}^{\leq 0},\Tbb_\mathbf{t}^{\geq 1})$ in $\Dcal$
defined as follows:
\begin{align*}
	\Tbb_\mathbf{t}^{\leq 0}&=\{X\in \mathcal D: H^0_\mathbb{T}(X)\in \mathcal T, \ H^0_\mathbb{T}(X[k])=0, \forall k>0\}\\
	\Tbb_\mathbf{t}^{\geq 1}&=\{X\in \mathcal D: H^0_\mathbb{T}(X)\in \mathcal F, \
	H^0_\mathbb{T}(X[k])=0, \forall k<0\}.
\end{align*}
This $t$-structure is called the \textbf{HRS-tilt of $\mathbb{T}$ with respect to $\mathbf{t}$} and its heart is
\[\mathcal H_\mathbf{t}:=\{X\in \mathcal D: H^0_\mathbb{T}(X[-1])\in \mathcal F, H^0_\mathbb{T}(X)\in \mathcal T,
	H^0_\mathbb{T}(X[k])=0\text{ for every } k\neq -1,0\}.\]
The pair $(\Fcal[1],\Tcal)$ is a torsion pair in $\Hcal_\mathbf{t}$.
\end{thm}

\begin{rmk}\label{TTF in HRS}
In the theorem above, if $\Hcal_\mathbf{t}$ is a Grothendieck
category and $\Tcal$ is a hereditary torsion class in $\Hcal_\mathbb{T}$, it can moreover be shown that $\Tcal$ is in
fact a TTF class in $\Hcal_\mathbf{t}$. Since $\Hcal_\mathbf{t}$ is Grothendieck and $\Tcal$ is a
torsionfree class, $\Tcal$ is closed under coproducts and it only remains to see that it is closed
under quotients. If $f\colon X\to Z$ is an epimorphism in $\Hcal_\mathbf{t}$,
$H^{-1}_\mathbb{T}(Z)$ lies in $\Fcal$ (because $Z$ lies $\Hcal_\mathbf{t}$) and, simultaneously,
$H^{-1}_\mathbb{T}(Z)$ is a subobject (in $\Hcal_\mathbb{T}$) of $\mathsf{Ker}(f)$, which is an object of
$\Tcal$. Thus, we have $H^{-1}_\mathbb{T}(Z)=0$ and $Z$ lies in $\Tcal$.
\end{rmk}

There are many results in the literature linking the properties of a torsion pair $\mathbf{t}$ 
and those of the HRS-tilt and its associated heart. We recall the following.

\begin{thm}\label{Saorin}
Let $R$ be a ring and let $\mathbb{T}$ be a nondegenerate $t$-structure in $\D(R)$
satisfying the equivalent conditions (1--3) of Theorem
\ref{thm:$t$-structure-types}. Suppose that $\mathbf{t}=(\Tcal,\Fcal)$ is a torsion pair in
the heart $\Hcal_\mathbb{T}$ of $\Tbb$, and let $\mathbb{T}_\mathbf{t}$ be the associated HRS-tilt, with heart
$\Hcal_\mathbf{t}$. Then:
\begin{enumerate}
\item\cite[Proposition 6.1]{SSV} $\Tbb_\mathbf{t}$ satisfies the equivalent conditions of Theorem
	\ref{thm:$t$-structure-types} if and only if $\mathbf{t}$ is of finite type in $\Hcal_\mathbb{T}$. 
\item\cite[Proposition 6.4]{SSV} If $\Tbb_\mathbf{t}$ is compactly generated, then $\mathbf{t}$ is
	generated by a set of finitely presented objects in $\Hcal_\mathbb{T}$. 

\item If $R$ is noetherian and $\mathbb{T}$ is intermediate and restrictable, then the following are equivalent:
\begin{enumerate}
\item $\mathbf{t}$ is of finite type and it is restrictable;
\item $\mathbb{T}_\mathbf{t}$ is restrictable and $\Hcal_\mathbf{t}$ is Grothendieck;
\item $\Hcal_\mathbf{t}$ is a locally coherent Grothendieck category with $\mathsf{fp}(\Hcal_\mathbf{t})=\Hcal_\mathbf{t}\cap \D^b(\mathsf{mod}(R))$
\end{enumerate}
\end{enumerate}
\end{thm}
\begin{proof}
We comment on item (3). This is a small generalisation of \cite[Theorem 5.2]{S}. The arguments
comparing between the restrictability of $\mathbf{t}$ and the restrictability of $\Tbb$ follow
analogously to those in the proof of \cite[Proposition 5.1]{S}. The equivalence with item (c)
follows from Theorem \ref{MZ}.
\end{proof}

\subsection{Derived equivalences}
It is quite natural to ask whether the derived category of a heart $\Hcal$ in a triangulated
category $\Dcal$ is (triangle) equivalent to $\Dcal$. This issue was first addressed in \cite{BBD},
where a functor between the bounded derived category of $\Hcal$ and $\Dcal$ was built (provided that
$\Dcal$ is \textit{nice enough} --- which $\D(R)$ certainly is, for any ring $R$). A similar
construction for the unbounded derived category is the subject of \cite{V}; the second statement of the
following theorem has been translated from the language of derivators, and specialised to our
setting.

\begin{thm}\label{prop:bound2unbound}
Let $R$ be a ring and $\Tbb$ an intermediate $t$-structure in $\D(R)$ with heart $\Hcal_\mathbb{T}$. 
\begin{enumerate}
\item\cite[\S 3.1]{BBD} The inclusion $\Hcal_\mathbb{T}\hookrightarrow \D^b(R)$ extends to a triangle
	functor $\realb_\Tbb\colon \D^b(\Hcal_\mathbb{T})\to \D^b(R)$, called a \textbf{bounded realisation
	functor}. This functor induces isomorphisms
	\[ \Hom_{\Hcal_\Tbb}(X,Y)\simeq\Hom_{\D(R)}(X,Y)\quad \text{and}\quad \Ext_{\Hcal_\Tbb}^1(X,Y)\simeq
	\Hom_{\D(R)}(X,Y[1])\]
	for all $X,Y$ of ${\Hcal_\Tbb}$.  Moreover, $\realb_\Tbb$ is an equivalence if and only if for
	every $n>1$ we also have
	\[\Ext_{\Hcal_\Tbb}^n(X,Y)\simeq\Hom_{\D(R)}(X,Y[n]).\]
\item\cite[Theorem B]{V} If $\Tbb$ is smashing and ${\Hcal_\Tbb}$ is Grothendieck (in other
	words, if $\Tbb$ is a cosilting $t$-structure), then the inclusion ${\Hcal_\Tbb}\hookrightarrow\D(R)$
	extends to a triangle functor $\real_\Tbb\colon\D({\Hcal_\Tbb})\to \D(R)$, called an \textbf{(unbounded)
	realisation functor}, which restricts to a bounded realisation functor $\realb_\Tbb\colon \D^b({\Hcal_\Tbb})\to \D^b(R)$. Moreover, if $\realb_\Tbb$ is a triangle equivalence, then so is $\real_\Tbb$.
\end{enumerate}
\end{thm}
\begin{proof}
In addition to the references provided, we would like to point out that {\cite[Theorem B]{V}}
applies to our context precisely because the fact that the $t$-structure is intermediate guarantees
that products in the heart have finite homological dimension (see \cite[Lemma 7.6]{V}). Finally we
comment on the fact that a bounded equivalence induces an unbounded equivalence in the setting
described above. Indeed, by \cite[Corollary 5.2]{PV}, an intermediate cosilting $t$-structure
inducing bounded derived equivalence is cotilting. The result then follows by \cite[Theorem 7.9]{V}.
\end{proof}

\begin{dfn}\label{def induces der eq}
We say that an intermediate $t$-structure $\mathbb{T}$ in $\D(R)$ with heart $\Hcal_\mathbb{T}$ \textbf{induces a (bounded)
derived equivalence} if $\real_\Tbb$ (respectively, $\realb_\Tbb$) is an equivalence. In that case
we say that $\Hcal_\mathbb{T}$ and $\mathsf{Mod}(R)$ are \textbf{(bounded) derived equivalent}.
\end{dfn}

\section{Preliminaries: commutative noetherian rings}

In this section, $R$ denotes a commutative noetherian ring. The set of prime ideals of $R$,
partially ordered by inclusion, will be denoted by $\Spec(R)$. For an ideal $I\leq R$, we write 
\[\spcl(I) :=\{\mathfrak p\in\Spec(R)\colon I\subseteq\mathfrak p\}\quad \text{and}\quad
\gncl(I):=\{\mathfrak p\in\Spec(R)\colon \mathfrak p\subseteq I\}\]

The set $\Spec(R)$ has a natural topology, whose closed subsets are the $\spcl(I)$ for all ideals
$I\leq R$. This is called the \textbf{Zariski topology} on $\Spec(R)$. This topological space turns
out to encode significant information concerning the representation theory of $R$.

\begin{dfn}
A subset $\mathcal P$ of $\Spec(R)$ is said to be \textbf{specialisation-closed} if for any $\mathfrak p$ in $\mathcal P$ we have that $\spcl(\mathfrak p)$ is contained in
$\mathcal P$. Dually, the subset $\mathcal P$ is called \textbf{generalisation-closed} if for any $\mathfrak p$ in $\mathcal P$ we have that $\gncl(\mathfrak p)$
is contained in $\mathcal P$.
\end{dfn}

Note that the complement of a specialisation-closed subset is generalisation-closed and vice-versa. We will denote the
complement of a subset $\Pcal\subseteq \Spec(R)$ by $\Pcal^\co$. From their definition, specialisation-closed subsets
are (possibly infinite) unions of Zariski-closed subsets, and thus, generalisation-closed subsets are (possibly
infinite) intersections of Zariski-open subsets. For a family $\mathcal P\subseteq\Spec R$, its
\textbf{specialisation closure} is the smallest specialisation-closed set containing $\mathcal P$, namely
$\spcl(\mathcal P):=\bigcup_{\mathfrak p\in\mathcal P}\spcl(\mathfrak p)$.

\subsection{Supports}
Given a prime ideal $\mathfrak p$ of $R$, let $R_\mathfrak{p}$ denote the localisation of $R$ at the
complement of $\mathfrak{p}$ and $k(\mathfrak p)=R_{\mathfrak p}/\mathfrak p R_{\mathfrak p}$ the
residue field of $R$ at $\mathfrak p$. Consider the two left derived functors 
\[ -_{\mathfrak p}:=-\otimes_R R_{\mathfrak p}\colon \D(R)\longrightarrow \D(R)\qquad\text{and}\qquad
	-\otimes^\mathbb{L} k(\mathfrak{p})\colon \D(R)\longrightarrow \D(R),\]
Since $R_{\mathfrak p}$ is a flat $R$-module, for a complex $X$, $X_\mathfrak{p}:=X\otimes_R
R_{\mathfrak p}$ is the componentwise localisation of $X$ as an object of $\D(R)$. In particular, we
have $H^i(X_{\mathfrak p})\simeq H^i(X)_{\mathfrak p}$ for all $i$ in $\mathbb{Z}$.
\begin{dfn}
Let $R$ be a commutative noetherian ring. Given a complex $X$ in $\D(R)$, we define the following
subsets of $\Spec(R)$:
\begin{itemize} 
\item $\supp(X):=\{\mathfrak p\in\Spec(R)\colon X\otimes_R^{\mathbb L}k(\mathfrak p)\neq 0\}$,
	called the (small) \textbf{support} of $X$;
\item $\Supp(X):=\{\mathfrak p\in\Spec(R)\colon X_{\mathfrak p}\neq 0\}$, called the \textbf{big
	support} of $X$.
\end{itemize}
The (big) support of a subcategory $\Xcal$ of $\D(R)$ is the union of the (big) supports of the
objects in $\Xcal$. Since localisation at $\mathfrak p$ commutes with standard
cohomology, as noticed above, $\Supp(X)=\Supp(\coprod H^i(X))$.  This set is
therefore also called the \textbf{homological support} of $X$.
\end{dfn}

Recall that there is a bijection between $\Spec(R)$ and the set of isoclasses of indecomposable
injective $R$-modules (Matlis' Theorem). This assignment is given by sending a prime ideal
$\mathfrak{p}$ to the injective envelope of $R/\mathfrak{p}$ --- which we denote by
$E(R/\mathfrak{p})$. Moreover, since $R$ is noetherian, every injective $R$-module is a coproduct of
copies of such indecomposable injectives. The following lemma gathers some well-known statements
about support that we will use later.

\begin{lemma}\label{basics supp} Let $\mathfrak{p}$ be a prime ideal of a commutative noetherian ring $R$.
\begin{enumerate}
\item\begin{enumerate}
\item $\supp(k(\mathfrak p))=\{\mathfrak p\}=\supp(E(R/\mathfrak p))$;
\item $\supp(R/\mathfrak p)=\spcl(\mathfrak p)$;
\item $\supp(R_\mathfrak{p})=\gncl(\mathfrak p)$;
\end{enumerate}
\item For any $X$ in $\D(R)$, $\Supp(X)$ is specialisation closed;
\item\cite[pag.~158]{Foxby} For any $X$ in $\D(R)$, $\supp(X)\subseteq\Supp(X)$;
\item\cite[Proposition 2.8/Remark 2.9]{Foxby}\cite[Proposition 2.1/Remark 2.2]{CI} For any
	bounded below $X$ in $\D^{+}(R)$,
	$\supp(X)$ coincides with the set of prime ideals $\mathfrak{p}$ for which the module
	$E(R/\mathfrak{p})$ is a summand of a module appearing in the minimal K-injective resolution of $X$;
\item For any bounded below $X$ in $\D^+(R)$, we have $\spcl(\supp(X))=\Supp(X)$.
\end{enumerate}
\end{lemma}
\begin{proof}
We prove items (2) and (5) for the sake of completeness
(see \cite[Lemma 2.2]{BIK} for the analogous statements for modules).
For (2), note that the claim follows from \cite[Lemma 2.1]{BIK}, taking into account that the big support of a complex coincides with the union of the big supports of its cohomologies. 

For (5), notice that by (2) and (3) we already have that $\spcl(\supp(X))\subseteq\Supp(X)$. For the
converse inclusion,
let $X\to E(X)$ denote a minimal $K$-injective resolution of $X$ in
$\D(R)$ and let $\mathfrak{q}$ be a prime ideal of $R$.  Then $\mathfrak{q}$ is not in
$\spcl(\supp X)$ if and only if $\mathfrak{q}$ is not in $\spcl(\mathfrak{p})$ for any
$\mathfrak{p}$ in $\mathsf{supp}(X)$, i.e. if and only if, by \cite[Lemma 2.2]{BIK},
$\mathfrak{q}$ is not in $\spcl(\mathfrak{p})=\Supp(E(R/\mathfrak{p}))$ for any
$\mathfrak{p}$ in $\mathsf{supp}(X)$. Hence, by item (4), if $\mathfrak{q}$ does not lie in
$\spcl(\supp X)$ then $E(X)_\mathfrak{q}=0$ (or, equivalently, $X_\mathfrak{q}=0$)
in $\D(R)$.
\end{proof}

For a subset $\Pcal\subseteq\Spec(R)$, we write $\supp^{-1}(\Pcal)$ (respectively,
$\Supp^{-1}(\Pcal))$ for the subcategory of $\D(R)$ whose objects have small (respectively,
big) support contained in $\Pcal$. By item (3) above, $\Supp^{-1}(\Pcal)$ is contained in
$\supp^{-1}(\Pcal)$. If $\Pcal$ is specialisation-closed, then item (5) of the lemma above guarantees
that a bounded below complex belongs to $\supp^{-1}(\Pcal)$ if and only if it belongs to
$\Supp^{-1}(\Pcal)$, i.e.
\begin{equation}\label{small vs big}
\Pcal=\spcl(\Pcal)\quad\Longrightarrow\quad \supp^{-1}(\Pcal)\cap \D^+(R)=\Supp^{-1}(\Pcal)\cap \D^+(R)
\end{equation}
\subsection{Localising subcategories}
Hereditary torsion classes of $\Mod(R)$ are also called \textbf{localising subcategories}.
There is a well-known bijection (holding, in fact, for any ring) between localising subcategories of
$\Mod(R)$ and \textbf{Giraud subcategories} of $\Mod(R)$, i.e. subcategories whose
inclusion functor admits an exact left adjoint (\cite{St}). This bijection associates a localising subcategory
$\Tcal$ of $\Mod(R)$ to the Giraud subcategory $\Tcal^{\perp_{0,1}}$, whose objects (we
recall) are those $R$-modules $X$ such that for any $T$ in $\Tcal$
\[\Hom_R(T,X)=0=\Ext^1_R(T,X).\]

Localising subcategories of $\mathsf{Mod}(R)$ are completely characterised by their support as follows.

\begin{thm}\label{review torsion}
Let $R$ be a commutative noetherian ring. Then the following statements hold.
\begin{enumerate}
\item\cite{G} The assignment of support yields a bijection between localising subcategories of
	$\mathsf{Mod}(R)$ and specialisation-closed subsets of $\Spec(R)$. 
\item For a localising subcategory $\Tcal$ of $\mathsf{Mod}(R)$ we have
\begin{enumerate}
\item a prime $\mathfrak{p}$ lies in $\supp(\Tcal)$ if and only if $k(\mathfrak{p})$ lies in $\Tcal$;
\item for any $\mathfrak{p}$ in $\Spec(R)$, then $k(\mathfrak{p})$ lies in either $\Tcal$ or in $\Tcal^{\perp_{0,1}}$;
\item $\Tcal^\perp=\{M\in\mathsf{Mod}(R)\colon \mathsf{Ass}(M)\cap
	\supp(\Tcal)=\emptyset\}=\mathsf{Cogen}(\mathsf{supp}^{-1}(\supp(\Tcal)^\co)\cap \mathsf{Mod}(R))$.
\end{enumerate}
\item\cite[Lemma 4.2]{AH} A torsion pair in $\mathsf{Mod}(R)$ is hereditary if and only if it is of finite type.
\end{enumerate}
\end{thm}
\begin{proof}
These are well-known statements; we comment on those for which we could not find a direct reference, for sake of readability. The assertion (2)(a) follows from (1)
and Lemma \ref{basics supp}. For (2)(b), one needs to check that if $\mathfrak{p}$ does not lie in
$\supp(\Tcal)$, then $\Hom_R(T,k(\mathfrak{p}))=\mathsf{Ext}^1_R(T,k(\mathfrak{p)})=0$ for all $T$ in $\Tcal$. By
Lemma \ref{basics supp}, every injective module in the minimal injective resolution of
$k(\mathfrak{p)}$ is a coproduct of copies of $E(R/\mathfrak{p})$. The statement now follows from
the fact that $T$ has only maps to injective modules of the form $E(R/\mathfrak{q})$, with
$\mathfrak{q}$ in $\supp(\Tcal)$. Finally, (2)(c) follows from the fact that, since
$\Tcal$ is localising, $\Tcal^\perp$ is closed under injective envelopes and, thus, the torsion pair
is cogenerated by a set of injectives. Clearly, these are the ones not supported in
$\supp(\Tcal)$, as wanted.
\end{proof}

In the derived category, we can also parametrise certain subcategories in terms of their supports.

\begin{dfn}
Let $\Dcal$ be a triangulated category admitting arbitrary (set-indexed) coproducts. A subcategory
$\Lcal$ of $\Dcal$ is said to be \textbf{localising} if $\Lcal$ is a triangulated subcategory closed
under coproducts. A localising subcategory $\Lcal$ is furthermore said to be \textbf{smashing} if
$\Lcal^\perp$ is closed under coproducts.
\end{dfn}

\begin{thm}\cite{N}\label{Neeman}
Let $R$ be a commutative noetherian ring. Then, the following statements hold.
\begin{enumerate}
\item The assignment of support yields a bijection between localising subcategories of $\D(R)$ and
the power set of $\Spec(R)$. Moreover, this bijection restricts to a bijection between smashing
subcategories of $\D(R)$ and the set of specialisation-closed subsets of $\Spec(R)$. 
\item For a localising subcategory $\Lcal$ of $\D(R)$ we have 
\begin{enumerate} 
\item a prime $\mathfrak{p}$ lies in $\supp(\Lcal)$ if and only if $k(\mathfrak{p})$ lies in $\Lcal$;
\item for any $\mathfrak{p}$ in $\Spec(R)$, then $k(\mathfrak{p})$ lies either in $\Lcal$ or in $\Lcal^{\perp}$;
\item $\Lcal$ is the smallest localising subcategory containing $\{k(\mathfrak{p})\colon
	\mathfrak{p}\in \supp(\Lcal)\}$;
\item $(\Lcal,\Lcal^{\perp})$ is a $t$-structure.
\end{enumerate}
\end{enumerate}
\end{thm}

Notice the parallel, albeit with some subtle differences, between the abelian and the derived
classification results. The following result summarises the relation between the two theorems above.

\begin{prop}
Let $R$ be a commutative noetherian ring, $V$ a specialisation-closed subset of $\Spec(R)$ and
$\Lcal=\supp^{-1}(V)$ the associated smashing subcategory of $\D(R)$. Then the localising
subcategory $\Tcal$ of $\mathsf{Mod}(R)$ associated to $V$ is $\Lcal\cap \mathsf{Mod}(R)$ and,
moreover, $\Tcal^{\perp_{\geq 0}}=\Lcal^{\perp}\cap \mathsf{Mod}(R)$.
\end{prop}

\begin{rmk}
Note that for a hereditary torsion class in $\mathsf{Mod}(R)$, it can be shown that
$\Tcal^{\perp_{\geq 0}}=\Tcal^{\perp_{0,1}}$ if and only if $\Tcal$ is a perfect torsion class, i.e.
if and only if the associated Gabriel topology is perfect. In \S \ref{sub:perfect} we will prove
this fact and  make use
of it to obtain a complete classification of hereditary torsion pairs in the HRS-tilt of
$\Mod(R)$ with respect to a perfect torsion pair.
\end{rmk}

\subsection{$t$-structures}

There is a wealth of knowledge about some classes of $t$-structures in the derived category of a
commutative noetherian ring. Once again subsets of the spectrum play a crucial role in the
characterisation of (compactly generated) $t$-structures.

\begin{dfn}
Let $R$ be a commutative noetherian ring. A function $\varphi$ from $\mathbb{Z}$ to the power set of
$\Spec(R)$ is said to be an \textbf{sp-filtration of $\Spec(R)$} if $\varphi$ is
a decreasing function between posets (i.e. if for all integers $n$, $\varphi(n)\supseteq \varphi(n+1)$)
and $\varphi(n)$ is specialisation-closed, for all $n$.
\end{dfn}

The following theorem concerns the classification and properties of compactly generated
$t$-structures over commutative noetherian rings.

\begin{thm}\label{t-st comm}\cite[Theorem 4.10]{AJS} \cite[Theorem 1.1]{HN}
Let $R$ be a commutative noetherian ring. The following are equivalent for a nondegenerate
$t$-structure $\mathbb{T}=(\Ucal,\Vcal)$ in $\D(R)$.
\begin{enumerate}
\item $\mathbb{T}$ is compactly generated;
\item $\mathbb{T}$ is smashing with Grothendieck heart (and the equivalent conditions of
Theorem \ref{thm:$t$-structure-types});
\item There is an sp-filtration of $\Spec(R)$ for which 
	\begin{align*}
		\Ucal&=\{X\in \D(R)\colon \Supp(H^0(X[n]))\subseteq \varphi(n), \forall n\in\mathbb{Z}\}\\
		\Vcal&=\{X\in \D(R)\colon \mathbb{R} \Gamma_{\varphi(n)}(X)\in\Dbb^{\geq n+1},
		\forall n\in\mathbb{Z}\},
	\end{align*}
	where $\Gamma_V$ denotes the (left exact) torsion radical of the hereditary torsion pair
	$(\Supp^{-1}(V),\Fcal_V)$ of $\Mod(R)$, for a specialisation-closed set $V$ (see Theorem \ref{review
	torsion}(1)).
\end{enumerate}
\end{thm}

It is easy to see that a $t$-structure as in the theorem above is nondegenerate if and only if
the associated sp-filtration $\varphi$ satisfies $\cup_{n\in\mathbb{Z}}\varphi(n)=\Spec(R)$ and
$\cap_{n\in\mathbb{Z}}\varphi(n)=\emptyset$. Moreover, such a $t$-structure is intermediate if and
only if there are integers $a< b$ such that $\varphi(a)=\Spec(R)$ and
$\varphi(b)=\emptyset$. In this case $\varphi$ will therefore be called \textbf{intermediate}.

Combining the results above, we observe the following useful statement.

\begin{cor}\label{cor:finite-types}
Let $R$ be a commutative notherian ring, and $\Hcal_\Tbb$ be the heart of a nondegenerate compactly generated
$t$-structure $\Tbb$ in $\D(R)$. Then, a torsion pair $\mathbf{t}=(\Tcal,\Fcal)$ in $\Hcal_\Tbb$ is of finite
type if and only if it is generated by a set of finitely presented objects of $\Hcal_\Tbb$.
\end{cor}
\begin{proof}
By Lemma \ref{CB+H}(1), we only need to prove one implication. If $\mathbf{t}$ is
of finite type in $\Hcal_\Tbb$, Theorem \ref{Saorin}(1) shows that the HRS-tilt of $\Tbb$
with respect to $\mathbf{t}$ is a smashing $t$-structure with a Grothendieck heart and, thus, by Theorem \ref{t-st comm} it is compactly generated. The result then follows from \ref{Saorin}(2).
\end{proof}

\begin{rmk}
Note that in the above corollary, $\Hcal_\Tbb$ is not necessarily locally coherent --- although, as shown
in \cite{SS}, it is locally finitely presented. 
\end{rmk}

\section{Hereditary torsion pairs in Grothendieck hearts} 
In this section we discuss hereditary torsion pairs in a given Grothendieck heart in the derived
category of a commutative noetherian ring. Throughout, once again $R$ will denote a commutative
noetherian ring.

\subsection{A characterisation by support}
We begin by showing that hereditary torsion classes in the heart of a smashing nondegenerate
$t$-structure of $\D(R)$ are completely determined by their support in $\Spec(R)$.

\begin{prop}\label{some properties}
Let $R$ be a commutative noetherian ring and let $\mathbb{T}=(\Ucal,\Vcal)$ be a nondegenerate
$t$-structure in $\D(R)$ with heart $\Hcal_\mathbb{T}$ and cohomological functor $H^0_\mathbb{T}\colon
\D(R)\longrightarrow\Hcal_\mathbb{T}$. If $\mathbf{t}=(\Tcal,\Fcal)$ is a hereditary torsion pair in $\Hcal_\Tbb$,
then:
\begin{enumerate}
	\item for each $\mathfrak{p}$ in $\Spec(R)$, there is an integer $n_\mathfrak{p}$ for which $k(\mathfrak{p})[n_\mathfrak{p}]$ lies in $\Hcal_\Tbb$;
	\item $\supp(\Hcal_\Tbb)=\Spec(R)$;
	\item if $\mathbb{T}$ is, in addition, smashing, then
	\begin{enumerate}
		\item $\Lcal_\mathbf{t}:=\{X\in \D(R)\colon H_\mathbb{T}^0(X[i])\in\Tcal\text{ for every
		}i\in\Zbb\}$ is a localising subcategory of $\D(R)$;
		\item $\supp(\Lcal_\mathbf{t})=\supp(\Tcal)=\{\mathfrak{p}\in\Spec(R)\colon k(\mathfrak{p})[n_\mathfrak{p}]\in\Tcal\}$
		\item $\Lcal_\mathbf{t}$ is the smallest localising subcategory containing $\Tcal$;
		\item $\Lcal_\mathbf{t}^{\perp}\cap\Hcal_\Tbb\subseteq \Tcal^{\perp_{0,1}}$;
		\item for each $\mathfrak{p}$ in $\Spec(R)$, $k(\mathfrak{p})[n_\mathfrak{p}]$ lies
			in $\Tcal$ or $k(\mathfrak{p})[n_\mathfrak{p}]$ lies in the Giraud subcategory
			$\Tcal^{\perp_{0,1}}$;
		\item $\supp^{-1}(\supp(\Tcal))\cap \Hcal_\Tbb=\Tcal$, i.e. $\Tcal$ is completely determined by its support.
	\end{enumerate}
\end{enumerate}
\end{prop}
\begin{proof}
(1) It is shown in \cite[Lemma 2.7]{HN} that the following two subsets form a partition of $\mathbb{Z}$:
\[A(\mathfrak{p}):=\{a\in\mathbb{Z}\colon k(\mathfrak{p})\in\Ucal[a]\}\qquad\text{and}\qquad
	B(\mathfrak{p}):=\{b\in\mathbb{Z}\colon k(\mathfrak{p})\in\Vcal[b]\}.\]
Since $\mathbb{T}$ is nondegenerate, this is a nontrivial partition. Moreover, since $\mathcal U$
(respectively, $\mathcal V$) is closed under positive (respectively, negative) shifts, if $m\leq
n\in A(\mathfrak{p})$ then $m\in A(\mathfrak{p})$ (respectively, if $m\geq n\in B(\mathfrak{p})$ then
$m\in B(\mathfrak{p})$). Hence, $A(\mathfrak{p})$ has a maximum, say $\alpha$, and $B(\mathfrak{p})$
has a minimum, say $\beta$. Since these sets form a partition of $\mathbb{Z}$, we conclude that 
$\beta=\alpha+1$ and, thus, $k(\mathfrak{p})$ lies in
$\Ucal[\alpha]\cap\Vcal[\alpha+1]=\Hcal_\Tbb[\alpha]$. In other words, we have that
$k(\mathfrak{p})[-\alpha]$ lies in $\Hcal_\Tbb$, so it suffices to take $n_\mathfrak{p}=-\alpha$.

(2) Since $\supp(k(\mathfrak{p})[n])=\{\mathfrak{p}\}$ for any integer $n$, it follows from (1) that
$\supp(\Hcal_\Tbb)=\Spec(R)$.

(3)(a) Since $\Tcal$ is closed under subobjects, extensions and quotient objects, it is easy to see
that $\Lcal_\mathbf{t}$ is a triangulated subcategory. Furthermore, since $\mathbb{T}$ is smashing,
$H^0_\mathbb{T}$ commutes with coproducts and, thus, the fact that $\Tcal$ is closed under coproducts
allows us to conclude that $\Lcal_\mathbf{t}$ is a localising subcategory. 

(3)(b) Let us denote by $\Pcal$ the subset of $\Spec(R)$ consisting of the prime ideals
$\mathfrak{p}$ for which $k(\mathfrak{p})[n_\mathfrak{p}]$ lies in $\Tcal$. We prove our statement
by showing that
\[\supp(\Lcal_\mathbf{t})\subseteq \Pcal \subseteq \supp(\Tcal)\subseteq \supp(\Lcal_\mathbf{t}).\]
Let $\mathfrak{p}$ be a prime ideal of $R$. If $\mathfrak{p}$ lies in $\supp(\Lcal_\mathbf{t})$ then
$k(\mathfrak{p})$ lies in $\Lcal_\mathbf{t}$ (see Theorem \ref{Neeman}) and, thus,
$k(\mathfrak{p})[n_\mathfrak{p}]$ lies in $\Tcal$, i.e. $\mathfrak{p}$ lies in $\Pcal$. If
$\mathfrak{p}$ lies in $\Pcal$, since $\supp(k(\mathfrak{p})[n_\mathfrak{p}])=\{\mathfrak{p}\}$,
then $\mathfrak{p}$ lies in $\supp(\Tcal)$.
 Finally, if $\mathfrak{p}$ lies in
$\supp(\Tcal)$, since $\Tcal$ is contained in $\Lcal_\mathbf{t}$, it follows that $\mathfrak{p}$
lies in $\supp(\Lcal_\mathbf{t})$.

(3)(c) Since $\Tcal$ is contained in $\Lcal_\mathbf{t}$, the smallest localising subcategory
containing $\Tcal$ must be contained in $\Lcal_\mathbf{t}$. Conversely, if $\Lcal$ is an arbitrary
localising subcategory containing $\Tcal$, then $\supp(\Lcal)$ must contain
$\supp(\Tcal)=\supp(\Lcal_\mathbf{t})$ and, thus, $\Lcal$ must contain $\Lcal_\mathbf{t}$. 

(3)(d) Given $X$ in $\Lcal_\mathbf{t}^\perp\cap\Hcal_\Tbb$ and $T$ in $\Tcal$ (and, thus, in
$\Lcal_\mathbf{t}$), we have that $\Hom_{\D(R)}(T,X)=0$ and
$\Ext_{\Hcal_\Tbb}^1(T,X)\cong\Hom_{\D(R)}(T[-1],X)=0$ since $T[-1]$ lies also in $\Lcal_\mathbf{t}$. Thus,
$X$ lies in $\Tcal^{\perp_{0,1}}$. 

(3)(e) Let $\mathfrak{p}$ be an arbitrary prime ideal of $R$ and consider the object
$k(\mathfrak{p})[n_\mathfrak{p}]$ of $\Hcal_\Tbb$. By Theorem~\ref{Neeman}, this object either lies in
$\Lcal_\mathbf{t}\cap \Hcal_\Tbb=\Tcal$ or in $\Lcal_\mathbf{t}^\perp\cap \Hcal_\Tbb\subseteq
\Tcal^{\perp_{0,1}}$.

(3)(f) From (3)(c) and Theorem \ref{Neeman}, it follows that that
$\supp^{-1}(\supp(\Tcal)))=\Lcal_\mathbf{t}$. By definition of $\Lcal_\mathbf{t}$, we
have $\Lcal_\mathbf{t}\cap\Hcal_\Tbb=\Tcal$, thus proving our claim.
\end{proof}

Note that, in particular, it follows that the hereditary torsion classes of the heart of
a nondegenerate smashing $t$-structure form a set. Item (3)(f) of the previous proposition motivates the
following definition.

\begin{dfn}
Let $R$ be a commutative noetherian ring and $\Tbb$ a nondegenerate smashing $t$-structure in
$\D(R)$, with heart $\Hcal_\Tbb$. A set $U\subseteq\Spec(R)$ is called a
\textbf{$\Hcal_\Tbb$-support} if it is the support of a hereditary torsion class in $\Hcal_\Tbb$.
This torsion class will then be $\supp^{-1}(U)\cap\Hcal_\Tbb$.
\end{dfn}

As a side corollary of the proposition above, we deduce a relation between the support of a complex
and that of its cohomologies with respect to a smashing $t$-structure.

\begin{cor}
Let $R$ be a commutative noetherian ring, and let $\Tbb$ be a nondegenerate smashing $t$-structure
in $\D(R)$, with heart $\Hcal_\Tbb$ and cohomology functor $H^0_\Tbb$. Let $U\subseteq\Spec(R)$ be a
$\Hcal_\Tbb$-support: then, for every object $X$ of $\D(R)$, we have
\[\supp(X)\subseteq U \  \text{ if and only if } \ \supp(H^0_\Tbb(X[n]))\subseteq U \ \forall n\in\Zbb\]
\end{cor}
\begin{proof}
This is direct consequence of items (3.a) and (3.b) of Theorem \ref{some properties}.
\end{proof}

\begin{rmk}
Note that this corollary recovers and extends the known relation (see \cite[Corollary 5.3]{BIK})
\[	\spcl(\supp(X))=\spcl(\supp(\bigoplus\nolimits_{n\in\mathbb{Z}}H^0(X[n])))\]
by taking $\Tbb=\Dbb$ the standard $t$-structure, and using that both sides of the equation are
$\Mod(R)$-supports (i.e. specialisation-closed subsets).
\end{rmk}

The following theorem provides some examples of $\Hcal_\Tbb$-supports, for some particular kinds of
hearts.

\begin{thm}\label{classification of tp in heart}
Let $R$ be a commutative noetherian ring and let $\mathbb{T}=(\Ucal,\Vcal)$ be an intermediate
compactly generated $t$-structure in $\D(R)$ with heart $\Hcal_\Tbb$. The following statements hold.
\begin{enumerate}
\item If $V$ is specialisation closed, then
\begin{enumerate} 
\item $\Tcal_V:=\supp^{-1}(V)\cap\Hcal_\Tbb$ is a hereditary torsion class in $\Hcal_\Tbb$;
\item if, additionally, $\Tbb$ induces a derived equivalence, then 
	$\mathbf{t}_V=(\Tcal_V,\Tcal_V^{\perp})$ is a torsion pair of finite type and $\Tcal_V^\perp=
	\Cogen(\supp^{-1}(V^\co)\cap\Hcal_\Tbb)$.
\end{enumerate}
\item If $\Tbb$ is restrictable, then
	for any hereditary torsion pair of finite type $\mathbf{t}=(\Tcal,\Fcal)$ in $\Hcal_\Tbb$ we have that
	$\supp(\Tcal)$ is specialisation closed.
\end{enumerate}
\end{thm}
\begin{proof}
(1)(a) We first show that $\Tcal_V:=\supp^{-1}(V)\cap\Hcal_\Tbb$ is a hereditary torsion class in
$\Hcal_\Tbb$, whenever $V$ is a specialisation closed subset of $\Spec(R)$.  First note that, since $\Tbb$ is intermediate, $\Hcal_\Tbb$ is contained in $\D^b(R)$, and
since $V$ is specialisation closed, it follows from Lemma \ref{basics supp} that
$\supp^{-1}(V)\cap\Hcal_\Tbb=\Supp^{-1}(V)\cap\Hcal_\Tbb$. Since $\Tbb$ is compactly generated, it is homotopically smashing and, thus, both $\Ucal$ and $\Vcal$ are closed
under directed homotopy colimits. From \cite[Lemma 2.11]{HN} it follows that both $\Ucal$ and
$\Vcal$ are closed under $-\otimes_RR_\mathfrak{p}$ and, therefore, $-\otimes_RR_\mathfrak{p}$ is
exact in $\Hcal_\Tbb$, for any $\mathfrak{p}$ in $\Spec(R)$. This shows that given a short exact sequence
in $\Hcal_\Tbb$ of the form
\[\begin{tikzcd}
	0 \arrow{r} &X \arrow{r} &Y \arrow{r} &Z \arrow{r} &0
\end{tikzcd}\]
we have that $Y\otimes_R R_\mathfrak{p}=0$ if and only if $X\otimes_R
R_\mathfrak{p}=0=Z\otimes_RR_\mathfrak{p}$. In other words, we have that
$\Supp(Y)=\Supp(X)\cup\Supp(Z)$ and, thus $\Supp(Y)$ is contained in $V$ if and only if both
$\Supp(X)$ and $\Supp(Z)$ are contained in $V$. This shows that $\Tcal_V$ is closed under
extensions, subobjects and quotient objects. Since it is also clearly closed under coproducts,
$\Tcal_V$ is a hereditary torsion class.

(1)(b) Suppose now that $\mathbb{T}$ induces a derived equivalence. In this case we know that there
is an isomorphism $\Ext_{\Hcal_\Tbb}^k(X,Y)\cong \Hom_{\D(R)}(X,Y[k])$ for any $X$ and $Y$ in $\Hcal_\Tbb$ and
$k\geq 0$. In particular, for a subcategory $\Scal$ of $\Hcal_\Tbb$, there is no ambiguity when
calculating the orthogonal $\Scal^{\perp_J}$: this $\Ext$-orthogonal subcategory in $\Hcal_\Tbb$ coincides
with the intersection with $\Hcal_\Tbb$ of the orthogonal computed in $\D(R)$.

We first show that $\Tcal^{\perp_{\geq 0}}_V=\supp^{-1}(V^\co)\cap\Hcal_\Tbb$. It
follows from \cite{N} that $\Lcal_V:=\supp^{-1}(V)$ is a smashing subcategory of $\D(R)$
and, thus, $\Bcal_V:=\Lcal_V^{\perp}$ is also localising with
$\supp(\Bcal_V)=V^\co$. Since, from Proposition \ref{some properties},
$\Lcal_V$ is the smallest localising subcategory containing $\Tcal_V$ and since
$\mathbb{T}$ induces a derived equivalence, it follows that
$\Bcal_V\cap\Hcal_\Tbb=\Tcal_V^{\perp_{\geq 0}}$. Finally, note that since $\Tcal_V$ is
hereditary, we have that $(\Tcal_V, \Tcal_V^\perp)=({}^\perp
E_V,\mathsf{Cogen}(E_V))$ for an injective object $E_V$ in $\Hcal$, and $E_V$ lies in
$\Tcal^{\perp_{\geq 0}}$. It then follows that
$\Tcal_V^\perp=\mathsf{Cogen}(\Bcal_V\cap\Hcal_\Tbb)=\mathsf{Cogen}(\supp^{-1}(V^\co)\cap\Hcal_\Tbb)$.

We now show that this torsion pair $\mathbf{t}=(\Tcal_V,\Tcal_V^{\perp})$ is of finite type.  Recall
that since $\Lcal_V$ is smashing, it is well-known (see \cite[Theorem 4.2]{KrauseInv} and
\cite[Proposition 6.3]{LV}) that $\Bcal_V$ is a definable subcategory and, thus, closed under
directed homotopy colimits \cite[Theorem 3.11]{L}. By \cite[Corollary 5.8]{SSV}, since $\mathbb{T}$
is homotopically smashing, every direct limit in $\Hcal_\mathbb{T}$ is a directed homotopy colimit
in $\D(R)$. Hence, $\Bcal_V\cap\Hcal_\Tbb=T^{\perp_{\geq 0}}$ is closed under direct limits in $\Hcal_\Tbb$.
Since direct limits are exact in $\Hcal_\Tbb$ and $\Tcal_V^\perp=\mathsf{Cogen}(\Tcal^{\perp_{\geq
0}}\cap\Hcal_\Tbb)$, we get that $\Tcal_V^\perp$ is closed under direct limits.

(2) By Theorem \ref{MZ}, since $\Tbb$ is restrictable, $\Hcal_\Tbb$ is locally coherent and $\fp(\Hcal_\Tbb)=\Hcal_\Tbb\cap D^b(\mod(R))$. Since $\mathbf{t}$ is a hereditary torsion pair of finite type,
it follows from Lemma~\ref{CB+H}(2) that $\Tcal=\varinjlim (\Tcal\cap \mathsf{fp}(\Hcal_\Tbb))$.
Let $\Lcal$ be the smallest localising subcategory of $\D(R)$ containing
$\Tcal\cap\mathsf{fp}(\Hcal_\Tbb)$. Clearly, $\Lcal$ is contained in the smallest localising subcategory
containing $\Tcal$, which we denote by $\Lcal_\mathbf{t}$. Since $\Lcal$ is the aisle of a
$t$-structure (namely $(\Lcal,\Lcal^\perp)$), $\Lcal$ is closed under directed homotopy colimits. As
above, since $\mathbb{T}$ is homotopically smashing, directed limits in $\Hcal_\Tbb$ are directed
homotopy colimits in $\D(R)$ and, thus, $\Tcal$ is contained in $\Lcal$, showing that
$\Lcal=\Lcal_\mathbf{t}$. Therefore, we have that
$\mathsf{supp}(\Tcal)=\mathsf{supp}(\Lcal_\mathbf{t})=\mathsf{supp}(\Lcal)$. Now, by assumption, the
$t$-structure $(\Lcal,\Lcal^\perp)$ is generated by all shifts of $\Tcal\cap\mathsf{fp}(\Hcal_\Tbb)$
which, by assumption is made of complexes in $\D^b(\mathsf{mod}(R))$. Now by \cite[Theorem
3.10]{AJS} this means that $\Lcal$ is compactly generated and, therefore, smashing. This shows, by
Theorem \ref{Neeman}, that $\mathsf{supp}(\Tcal)=\mathsf{supp}(\Lcal)$ is specialisation closed.
\end{proof}

Notice that the theorem above provides an immediate generalisation of Theorem \ref{review torsion}, which we will further simplify in Corollary \ref{tp finite type restrictable}.

\begin{cor}\label{cor her ft in some hearts}
Let $R$ be a commutative noetherian ring and let $\mathbb{T}$ be a restrictable and intermediate compactly generated $t$-structure in $\D(R)$ such that 
$\mathbb{T}$ induces a derived equivalence. Then there is a bijection between hereditary torsion
pairs of finite type in $\Hcal_\Tbb$ and specialisation closed subsets of $\Spec(R)$.
\end{cor}

In fact, we can be more precise about the support of a hereditary torsion pair of finite type for
nondegenerate compactly generated $t$-structures.

\begin{prop}\label{residue fields in heart}
Let $R$ be a commutative noetherian ring and let $\mathbb{T}=(\Ucal,\Vcal)$ be a nondegenerate compactly generated
$t$-structure in $\D(R)$ with heart $\Hcal_\Tbb$ and  associated sp-filtration $\varphi$.
\begin{enumerate}
\item  For each $\mathfrak{p}$ in $\Spec(R)$, let $\varphi_{max}(\mathfrak{p})$ denote the largest
	integer $n$ for which $\mathfrak{p}$ belongs to $\varphi(n)$. Then we have
	$n_\mathfrak{p}=-\varphi_{max}(\mathfrak{p})$ and, in particular, if $\mathfrak{p}$ is contained in
	a prime $\mathfrak{q}$, then $n_\mathfrak{p}\geq n_\mathfrak{q}$.
\item If $\mathbf{t}=(\Tcal,\Fcal)$ is a hereditary torsion
	pair of finite type in $\Hcal_\Tbb$, then there is an sp-filtration $\psi$ such that
	$\varphi(j+1)\subseteq \psi(j)\subseteq \varphi(j)$ for all $j$ in $\mathbb{Z}$ and 
	\[\mathsf{supp}(\Tcal)=\bigcup_{j\in\mathbb{Z}} \left[\psi(j)\setminus \varphi(j+1)\right].\]
\end{enumerate}
\end{prop}
\begin{proof}
(1) Given the cohomological description of $\Ucal$ (see Theorem \ref{t-st comm}), we know that the stalk complex
$k(\mathfrak{p})[-\varphi_{max}(\mathfrak{p})]$ lies in $\Ucal$ but
$k(\mathfrak{p})[-(\varphi_{max}(\mathfrak{p})+1)]$ does not. By \cite[Lemma 2.7]{HN}, this means that
$k(\mathfrak{p})[-\varphi_{max}(\mathfrak{p})-1]$ lies in $\Vcal$ and, therefore,
$k(\mathfrak{p})[-\varphi_{max}(\mathfrak{p})]$ lies in $\Vcal[1]\cap\Ucal=\Hcal_\Tbb$, as wanted. Since an
sp-filtration is a decreasing sequence of specialisation-closed subsets, we have that if $\mathfrak{p}$ is
contained in a prime $\mathfrak{q}$, then $\varphi_{max}(\mathfrak{p})\leq \varphi_{max}(\mathfrak{q})$
and, thus, $n_\mathfrak{p}\geq n_\mathfrak{q}$.

(2) Let $\mathbf{t}=(\Tcal,\Fcal)$ be a hereditary torsion pair of finite type in $\Hcal_\Tbb$. The
$t$-structure obtained by HRS-tilting $\mathbb{T}$ with respect to $\mathbf{t}$ is compactly
generated, since $\mathbf{t}$ is of finite type (see Theorems \ref{Saorin} and \ref{t-st comm}).
Thus, it is determined by an sp-filtration $\psi$ satisfying $\varphi(j+1)\subseteq \psi(j)\subseteq
\varphi(j)$. Let $\Ucal_\psi$ be the aisle of the $t$-structure associated to $\psi$.  Clearly, we
have that $\Tcal=\Ucal_\psi\cap\Hcal_\Tbb$. We need to check which shifted residue fields belong to
$\Tcal$ (following Proposition \ref{some properties}(3)(b)).  For any $\mathfrak{p}$ in $\Spec(R)$,
by (1) $k(\mathfrak p)[-j]$ lies in $\Hcal_\Tbb$ if and only if $\mathfrak p$ lies in
$\varphi(j)\setminus\varphi(j+1)$. Now, $k(\mathfrak p)[-j]$ lies in $\Tcal$ if and only if
$k(\mathfrak p)[-j]$ lies in $\Hcal_\Tbb\cap \Ucal_\psi$, i.e. $\mathfrak p$ lies in $\psi(j)\setminus
\varphi(j+1)$. Thus $\supp(\Tcal)$ coincides with the union of all such sets $\psi(j)\setminus
\varphi(j+1)$.
\end{proof}

\subsection{A complete classification of hereditary torsion pairs in a special case}\label{sub:perfect}

The previous section shows that when trying to classify the hereditary torsion pairs in the heart
$\Hcal_\Tbb$ of a nondegenerate compactly generated $t$-structure, one may equivalently describe the
corresponding $\Hcal_\Tbb$-supports. For example, by Theorem \ref{classification of tp in heart} we know that specialisation-closed sets are
often $\Hcal_\Tbb$-supports.
While the problem of classifying all $\Hcal_\Tbb$-supports remains, in general, open, we are able to provide a
complete classification for some hearts. These occur as HRS-tilts of $\Mod(R)$ at a perfect torsion pair.

\begin{rmk}\label{rmk:hrs-at-hereditary}
Notice that the HRS-tilt of $\Mod(R)$ at a hereditary torsion pair, corresponding to a specialisation-closed subset
$V\subseteq\Spec(R)$, is always compactly generated. Indeed, its aisle corresponds to the sp-filtration
$\cdots=\Spec(R)\supseteq V\supseteq \emptyset=\cdots$, with $V$ in degree $0$.
\end{rmk}

\begin{dfn}
A hereditary torsion pair $\mathbf{t}=(\Tcal,\Fcal)$ in a module category is said to be
\textbf{perfect} if the associated Giraud subcategory $\Tcal^{\perp_{0,1}}$ admits a right adjoint
to the inclusion functor.
\end{dfn}

Note that, by definition, the inclusion functor of a Giraud subcategory admits a left adjoint.
Therefore, the adjective \textbf{perfect} applied to a hereditary torsion pair guarantees that the
associated Giraud subcategory is, in fact, a (extension-closed) bireflective subcategory of
$\mathsf{Mod}(R)$. We learned the following useful property of perfect torsion pairs from a private communication with Lidia Angeleri H\"ugel and Ryo Takahashi. We include a proof for sake of completion.

\begin{lemma}\label{cosygyzy}
	Let $R$ be a ring and let $\mathbf{t}=(\Tcal,\Fcal)$ be a perfect torsion pair in $\Mod(R)$. Then
	$\Tcal^{\bot_{0,1}}=\Tcal^{\bot_{\geq0}}$.
\end{lemma}
\begin{proof}
	Since $\mathbf{t}$ is hereditary, $\Fcal=\Tcal^\perp$ is closed under injective envelopes, so
	$\Tcal^{\bot_{0,1}}$ is as well. Recall that a Giraud subcategory of $\Mod(R)$ is always an
	abelian subcategory (\cite[Proposition 1.3]{St}). Since $\mathbf{t}$ is perfect,
	the inclusion functor $\Tcal^{\bot_{0,1}}\hookrightarrow\Mod(R)$ has a
	right adjoint: therefore, cokernels taken in
	$\Tcal^{\bot_{0,1}}$ coincide with those in $\Mod(R)$. 
	This shows that $\Tcal^{\bot_{0,1}}$ is closed under cokernels in $\Mod(R)$. It follows that $\Tcal^{\bot_{0,1}}$ is closed under cosygyzies, and one concludes by
	dimension shifting.
\end{proof}

The following lemma gives examples of hearts $\Hcal$ with $\Hcal$-supports that are not specialisation-closed.

\begin{lemma}\label{generalisation-closed support}
Let $R$ be a ring and let $\mathbf{t}=(\Tcal,\Fcal)$ be a perfect torsion pair in $\mathsf{Mod}(R)$. Denote the corresponding Giraud subcategory by
$\Ccal:=\Tcal^{\perp_{0,1}}$. Let $\Hcal_\mathbf{t}$ be the heart of the HRS-tilt
at $\mathbf{t}$. The following statements hold:
\begin{enumerate}
	\item There is a TTF triple $(\Fcal[1],\Tcal,\Ccal[1])$ in $\Hcal_\mathbf{t}$;
	\item $\Ccal[1]$ itself is a hereditary torsion class in $\Hcal_\mathbf{t}$;
	\item If $R$ is commutative noetherian, then $\supp(\Ccal[1])=V^\co$. Hence, $V^\co$ is a
	 (generalisation-closed) $\Hcal_\mathbf{t}$-support.
\end{enumerate}
\end{lemma}
\begin{proof}
(1) The fact that $\Tcal$ is TTF class in $\Hcal_\mathbf{t}$ follows from Remark \ref{TTF in HRS}.
We only need to verify that the corresponding torsionfree class in $\Hcal_\mathbf{t}$, denoted by $\Fcal'$, coincides with $\Ccal[1]$. Since
$\Hom_{\Hcal_\mathbf{t}}(T,C[1])=\Hom_{\D(R)}(T,C[1])\simeq\Ext_R^1(T,C)=0$, for all $T$ in $\Tcal$
and $C$ in $\Ccal$, we have $\Ccal[1]\subseteq \Fcal'$. For the converse, let $X$ be an object in
$\Fcal'$, and consider the triangle
\[\begin{tikzcd}
	F_X[1] \arrow{r} &X \arrow{r} &T_X \arrow{r}{w} &F_X[2]
\end{tikzcd}\]
which corresponds to the short exact sequence in $\Hcal_\mathbf{t}$ given by the torsion pair
$(\Fcal[1],\Tcal)$. Since $\Fcal'$ is closed under subobjects in $\Hcal_\mathbf{t}$, for every $T$
in $\Tcal$ we have
$0=\Hom_{\Hcal_\mathbf{t}}(T,F_X[1])=\Hom_{\D(R)}(T,F_X[1])\simeq\Ext_R^1(T,F_X)$.  This shows that
in fact $F_X$ lies in $\Ccal$. Since the torsion pair $\mathbf{t}$ is perfect,
$\Ccal=\Tcal^{\perp_{\geq 0}}$ and, therefore, $0=\Ext_R^2(T_X,F_X)\simeq\Hom_{\D(R)}(T_X,F_X[2])$.
Since $w$ lies in the latter $\Hom$-space, we conclude that $w=0$ and that the triangle
above splits; hence $T_X=0$ and $X=F_X[1]$ lies in $\Ccal[1]$.

(2) Since $\Ccal[1]$ is a torsionfree class in $\Hcal_\mathbf{t}$, it
suffices to show that $\Ccal[1]$ is closed under cokernels of monomorphisms in $\Hcal_\mathbf{t}$.
Consider then $C$ and $C'$ in $\Ccal$ and a triangle
\[C[1] \longrightarrow C'[1] \longrightarrow X \longrightarrow C[2]\]
with $X$ in $\Hcal_\mathbf{t}$. Applying the functor $\mathsf{Hom}_{\D(R)}(T,-)$ for every $T$ in
$\Tcal$, since $\Hom_{\D(R)}(T,C'[1])\cong \Ext^1_R(T,C')=0$ and
$\Hom_{\D(R)}(T,C[2])\cong \Ext^2_R(T,C)=0$ (given that $\mathbf{t}$ is a perfect torsion pair), we
have that $\Hom_{\D(R)}(T,X)=0$. This shows that $X$ indeed lies in $\Fcal'=\Ccal[1]$.

(3) If $R$ is commutative noetherian, since $\supp(\Ccal[1])=\supp(\Ccal)$, it suffices to observe that $\supp(\Ccal)=V^\co$ (because $E(R/\mathfrak{p})$ lies in $\Ccal$ for all $\mathfrak{p}$ not in $V$).  The last assertion then follows from part (2).
\end{proof}

In fact, using this generalisation-closed $\Hcal_\mathbf{t}$-support one can construct many more which are not
specialisation-closed. Recall that torsion pairs can be ordered by inclusion of torsion classes. Under this
partial order, they form a lattice, where the meet is given by intersecting torsion classes and the join is
given by intersecting torsionfree classes. The following lemma recalls that this structure restricts
to hereditary torsion classes, and shows that it translates well via the assignment of support.

\begin{lemma}\label{lemma:cupcap of supports} Let $R$ be a commutative noetherian ring and let $\Hcal$ be the heart of a nondegenerate compactly
generated $t$-structure in $\D(R)$. Let $\mathbf{t}_1=(\Tcal_1,\Fcal_1)$,
$\mathbf{t}_2=(\Tcal_2,\Fcal_2)$ be hereditary torsion pairs in
$\Hcal$ with $\supp(\Tcal_1)=V_1$, $\supp(\Tcal_2)=V_2$. Then:
\begin{enumerate}
	\item $\mathbf{t}_3:=\mathbf{t}_1\wedge\mathbf{t}_2:=(\Tcal_3:=\Tcal_1\cap\Tcal_2,\Fcal_3)$
	is a hereditary torsion pair in $\Hcal$ with $\supp(\Tcal_3)=V_1\cap V_2$;
	\item $\mathbf{t}_4:=\mathbf{t}_1\vee\mathbf{t}_2:=(\Tcal_4,\Fcal_4:=\Fcal_1\cap\Fcal_2)$ is
	a hereditary torsion pair in $\Hcal$ with $\supp(\Tcal_4)=V_1\cup V_2$.
\end{enumerate}
In particular, intersections and unions of $\Hcal$-supports are again
$\Hcal$-supports.
\end{lemma}
\begin{proof}
(1) If $\Tcal_1,\Tcal_2$ are closed under subobjects, then clearly $\Tcal_3$ is as well, so
$\mathbf{t}_3$ is hereditary. The claim on the support follows from the second equality of item
(3)(b) of Proposition \ref{some properties}.

(2) Recall that $\Hcal$ is a Grothendieck category. Since both $\Fcal_1$ and $\Fcal_2$ are closed
under injective envelopes, then so is $\Fcal_4$: thus, $\mathbf{t}_4$ is hereditary. Now using items
(3)(b) and (3)(e) of Proposition \ref{some properties}
one proves the claim about support.
\end{proof}

We are now able to classify supports in the heart $\Hcal_\mathbf{t}$ of an HRS-tilt at a perfect torsion pair.

\begin{prop}\label{prop:ht-supports}
Let $R$ be a commutative noetherian ring, $\mathbf{t}=(\Tcal,\Fcal)$ a perfect torsion pair in $\mathsf{Mod}(R)$ with $\supp(\Tcal)=V$,
and $\Hcal_\mathbf{t}$ the associated heart by HRS-tilt. Then the $\Hcal_\mathbf{t}$-supports are
all the sets of the form $(W\cap V)\cup(W'\cap V^\co)$, for $W,W'\subseteq\Spec(R)$ specialisation
closed.
\end{prop}
\begin{proof}
We have already noted that $\Hcal_\mathbf{t}$  is the heart of an intermediate compactly generated
$t$-structure, so it follows from Theorem \ref{classification of tp in heart}(1)(a) that specialisation-closed
subsets are $\Hcal_\mathbf{t}$-supports. Since both $V$ and $V^\co$ are $\Hcal_\mathbf{t}$-supports
(see Lemma \ref{generalisation-closed support}), it follows from Lemma \ref{lemma:cupcap of supports} that the
subsets presented in the statement are indeed $\Hcal_\mathbf{t}$-supports. To prove the converse, let $U$ be a $\Hcal_\mathbf{t}$-support and let $\mathbf{t}'=(\Tcal',\Fcal')$
be the hereditary torsion pair in $\Hcal_\mathbf{t}$ with $\supp(\Tcal')=U$. 
We first show that
\begin{enumerate}
\item If $U\subseteq V$, then $U$ is specialisation-closed;
\item If $U\subseteq V^\co$, then $U=\spcl(U)\cap V^\co$.
\end{enumerate}

(1) If $U\subseteq V$, then we have that $\Tcal'\subseteq\Tcal\subseteq\Mod(R)$. We show that $\Tcal'$ is a
hereditary torsion class in $\Mod(R)$ as well, and so $U$ is specialisation-closed. Clearly $\Tcal'$ is closed
under extensions and coproducts in $\D(R)$, and thus it is so in $\Mod(R)$ as well. Let now 
\[\begin{tikzcd}
(\ast)\colon &0 \arrow{r} &X \arrow{r} &Y \arrow{r} &Z \arrow{r} &0
\end{tikzcd}\]
be a short exact sequence in $\Mod(R)$, with $Y$ in $\Tcal'\subseteq\Tcal$. Since $\Tcal$ is a
hereditary torsion class in $\Mod(R)$, both $X$ and $Z$ belong to $\Tcal\subseteq\Hcal_\mathbf{t}$,
so that $(\ast)$ is a short exact sequence in $\Hcal_\mathbf{t}$ as well. Now since $\Tcal'$ is a
hereditary torsion class in $\Hcal_\mathbf{t}$ we conclude that $X$ and $Z$ belong to $\Tcal'$, as wanted.

(2) The non-trivial inclusion is $U\supseteq\spcl(U)\cap V^\co$. Let $\Ccal$ denote the Giraud subcategory associated to $\mathbf{t}$, i.e. $\Ccal=\Tcal^{\perp_{0,1}}$. Given $\mathfrak
p$ in $U\subseteq V^\co$ and  $\mathfrak q$ in $V^\co$ such that $\mathfrak
p\subseteq\mathfrak q$, we will show that $\mathfrak q$ lies in $U$ as well.
Translating this in terms of objects of $\Hcal_\mathbf{t}$, consider the stalk complexes
$E(R/\mathfrak p)[1]$ and $E(R/\mathfrak q)[1]$ in $\Ccal[1]\subseteq\Hcal_\mathbf{t}$. By
assumption we have $E(R/\mathfrak p)[1]$ lies in $\Tcal'$, and we want to prove that $E(R/\mathfrak
q)[1]$ lies in $\Tcal'$ as well. Denote by $(\ast\ast)$ the torsion sequence of $E(R/\mathfrak
q)[1]$ with respect to $\mathbf{t}'$. Since, by Lemma \ref{generalisation-closed support}, $\Ccal[1]$ is a hereditary torsion class in
$\Hcal_\mathbf{t}$, we deduce that $(\ast\ast)$ has all its terms in $\Ccal[1]$.  Therefore,
applying a shift to it, we obtain the exact sequence of modules (in solid arrows) with $T$ in $\Tcal'[-1]$ and $F$ in $\Fcal'[-1]$.
\[\begin{tikzcd}
	(\ast\ast)[-1]\colon &0 \arrow{r} &T \arrow[dotted]{dr}\arrow{r} &E(R/\mathfrak q) \arrow[dotted, shift left]{d}\arrow{r} &F \arrow{r} &0 \\
	& & & E(T) \arrow[dotted, shift left]{u}
\end{tikzcd}\]

In $\Mod(R)$, consider then the injective envelope $E(T)$ of $T$: by injectivity, we get the two
dotted vertical arrows in the diagram above. Moreover, since the morphism $T\to E(T)$ is left
minimal, we conclude that $E(T)$ is a direct summand of the indecomposable module $E(R/\mathfrak
q)$. Now, we use our hypothesis that $\mathfrak p\subseteq\mathfrak q$ to notice that there is a
nonzero morphism $E(R/\mathfrak p)[1]\to E(R/\mathfrak q)[1]$. Since the source of this morphism is in $\Tcal'$, the
target cannot be in $\Fcal'$, and therefore $T\neq0$. Then, we must have an isomorphism $0\neq E(T)\simeq
E(R/\mathfrak q)$, which means that
\[\{\mathfrak q\}=\supp(E(R/\mathfrak q))=\supp(E(T))\subseteq\supp(T)=\supp(T[1])\subseteq U.\]

Returning to the general case of an arbitrary $\Hcal_\mathbf{t}$-support $U$, note that by Lemma
\ref{lemma:cupcap of supports} and Lemma \ref{generalisation-closed support}, both $U\cap V$ and $U\cap V^\co$
are $\Hcal_\mathbf{t}$-supports. Set $W:=U\cap V$ and $W':=\spcl(U\cap V^\co)$: the first is
specialisation-closed by item (1) above, while the second is specialisation-closed by definition.
Now by item (2) above it follows that
$U=(U\cap V)\cup(U\cap V^\co)=(W\cap V)\cup(W'\cap V^\co)$.
\end{proof}

\begin{rmk}
Item (3) of Theorem \ref{review torsion} states that in $\Mod(R)$ hereditary torsion pairs coincide
with those of finite type. We will see in the next section (see Corollary \ref{thm:her-criterion})
that the hearts $\Hcal_\mathbf{t}$ considered in this subsection are all derived equivalent to
$\mathsf{Mod}(R)$. Then by Corollary \ref{cor her ft in some hearts} (since every torsion pair
$\mathbf{t}$ in $\Mod(R)$ is restrictable and, thus, so is the $t$-structure
$\mathbb{T}_\mathbf{t}$, by Theorem \ref{Saorin}), we conclude that in $\Hcal_\mathbf{t}$ hereditary
torsion pairs of finite type correspond bijectively to specialisation closed subsets of $\Spec(R)$.
Proposition \ref{prop:ht-supports} then shows that if $\mathbf{t}$ is perfect then, in general, not
every hereditary torsion pair is of finite type (since not all $\Hcal_\mathbf{t}$-supports are
specialisation-closed). 
\end{rmk}

We conclude this subsection with an illustrating example.

\begin{ex}\label{ex:Z}
Let $R$ be a commutative noetherian ring of Krull dimension 1. In this case, every hereditary
torsion pair is perfect (see \cite[Corollary 4.3]{K} and \cite[Corollary 4.10]{AMSTV}). Let $V$
denote the set of maximal ideals of $R$. It is, of course, a specialisation-closed subset of $\Spec(R)$;
denote by $\mathbf{t}=(\Tcal,\Fcal)$ the associated hereditary torsion pair in $\mathsf{Mod}(R)$.
Let $\Hcal_\mathbf{t}:=\Fcal[1]\ast \Tcal$ be the heart of the HRS-tilt of the standard
$t$-structure with respect to $\mathbf{t}$. Following Proposition \ref{prop:ht-supports}, the
$\Hcal_\mathbf{t}$-supports are the sets of primes of the form $( W \cap V)\cup( W'\cap V^\co)$, for specialisation-closed subsets $W$ and $W'$ of $\Spec(R)$. However, it is quite easy to see that, since $R$ has Krull dimension 1, any subset of $\Spec(R)$ is of this form. Interestingly, this means that in this case hereditary torsion pairs in
$\Hcal_{\mathbf{t}}$ are in bijection with localising subcategories of $\D(R)$ (not only
the smashing ones, as it happens with hereditary torsion pairs in $\Mod(R)$). Concretely, following items (3)(a) and (3)(b) of
Theorem \ref{some properties}, the bijection is
\[ \begin{array}{rcl}
	\{\textit{hereditary torsion pairs in
	}\Hcal_{\mathbf{t}}\}&\overset{1:1}{\longleftrightarrow}&
		\{\textit{localising subcategories of }\D(R)\} \\
	(\Tcal,\Fcal)&\longmapsto&\{X\in\D(R)\colon H_{\mathbf{t}}^0(X[i])\in\Tcal\text{ for
	every }i\in\Zbb\}\\
	\Hcal_{\mathbf{t}}\cap\Lcal &\longmapsfrom& \Lcal
\end{array}\]
where $H^0_{\mathbf{t}}\colon \D(R)\to \Hcal_{\mathbf{t}}$ is the cohomology functor. In
particular, all localising subcategories of $\D(R)$ admit a cohomological description with respect to
$\Hcal_{\mathbf{t}}$. We will later prove that $\Hcal_{\mathbf{t}}$ is derived equivalent to $\Mod(R)$ (as a consequence of Corollary \ref{der eq comm ring}). This means that we get different insights on the triangulated structure of this derived category, depending on the abelian category that we start with. 
\end{ex}

\section{Torsion pairs inducing derived equivalences}

For this section we leave the commutative noetherian setting to explore the following
general notion for an abelian category $\Acal$ that admits a derived category $\D(\Acal)$. 
\begin{dfn}
We say that a torsion pair $\mathbf{t}=(\mathcal T,
\mathcal F)$ in $\Acal$ \textbf{induces a (bounded) derived equivalence} if the associated HRS-tilt in $\mathsf{D}(\mathcal{A})$ induces a (bounded) derived equivalence (as in Definition \ref{def induces der eq}).
\end{dfn}

In order to understand when a torsion pair induces (bounded) derived equivalence, we will make use
of a criterion by \cite{CHZ} (see \S \ref{sub:CHZ-sequences}). Before doing that, we will recall some homological
tools.

\subsection{Zero Yoneda extensions}

Fix two objects $X$ and $Y$ in an abelian category
$\Acal$. Recall that the elements of the Yoneda group $\Ext_\Acal^n(X,Y)$ (i.e. \textbf{Yoneda $n$-extensions}) are exact
sequences of the form
\[ \varepsilon\colon\qquad 0\to Y\to Z_1\to \cdots\to Z_n\to X\to 0,\]
which we call \textbf{$n$-sequences}, up to the equivalence relation generated by pairs of
exact sequences $(\varepsilon,\varepsilon')$ such that there is a morphism
\begin{equation}\label{extequiv}\nonumber
\begin{tikzcd}
	\varepsilon\colon \arrow{d} &0 \arrow{r} &Y \arrow[equal]{d}\arrow{r} &Z_1
		\arrow{d}\arrow{r} &\cdots \arrow{d}\arrow{r} &Z_n \arrow{d}\arrow{r}& X
		\arrow[equal]{d}\arrow{r} &0 \\
	\varepsilon'\colon &0 \arrow{r} &Y \arrow{r} &Z_1' \arrow{r} &\cdots \arrow{r}& Z_n' \arrow{r} &X \arrow{r} &0
\end{tikzcd}
\end{equation}
making the diagram commute. This is called a \textbf{morphism of $n$-sequences}.
In other words, two $n$-sequences are equivalent if and only if there exists a zigzag of morphisms of $n$-sequences linking them. In particular, if $n=1$, two 1-sequences are equivalent if and only
if they are isomorphic. For $n>1$, on the other hand, $n$-sequences are clearly in bijection with complexes
\[Z:=0\to Z_1\to \cdots\to Z_n\to 0\]
with $Z_n$ in degree zero and which have $H^{-n+1}(Z)=Y$, $H^0(Z)=X$, and are acyclic in all other
degrees.  It is then easy to observe that a morphism of $n$-sequences translates to a
quasi-isomorphism of the associated complexes, and viceversa. Since isomorphisms in the derived
category $\D(\Acal)$ are zigzags of quasi-isomorphisms of complexes, it follows that two
$n$-sequences are equivalent (and therefore, they represent the same Yoneda $n$-extension) if and
only if the associated complexes are isomorphic in $\D(\Acal)$.

\begin{lemma}\label{homological algebra}
Let $X$ and $Y$ be objects of an abelian category $\Acal$. Let $n>1$ and consider an exact sequence
\[\varepsilon\colon\qquad 0\to Y\to Z_1\to Z_2\to \cdots\to Z_n\to X\to 0.\]
Let $Z$ be the associated complex as above. 
Then $\varepsilon$ represents the zero element in $\Ext_\Acal^n(X,Y)$ if and only if the following truncation
triangle in $\D(\Acal)$ 
\[\tau^{\leq -1}Z\longrightarrow Z\longrightarrow \tau^{\geq 0}Z\longrightarrow (\tau^{\leq
-1}Z)[1]\]
where $\tau^{\leq -1}$ and $\tau^{\geq 0}$ are the truncation functors for the standard t-structure in $\mathsf{D}(\mathcal A)$, is a split triangle.
\end{lemma}
\begin{proof}
The zero element of $\Ext_\Acal^n(X,Y)$ is the equivalence class of the
$n$-sequence
\[0\longrightarrow Y\overset{1_Y}{\longrightarrow}Y\longrightarrow 0\longrightarrow \cdots\longrightarrow
0\longrightarrow X\overset{1_X}{\longrightarrow} X\longrightarrow 0\]
and its associated complex is, clearly, the direct sum $X\oplus Y[n-1]$. By the considerations above,
$\varepsilon$ represents the zero Yoneda $n$-extension if and only if the associated complex $Z$ is
isomorphic to $X\oplus Y[n-1]$. This occurs if and only if, for any degree $-n+1\leq i< 0$, the truncation triangle
\[\begin{tikzcd}
	\tau^{\leq i} Z \arrow{d}{\simeq}\arrow{r} &Z \arrow{r} &\tau^{\geq i+1}Z
	\arrow{d}{\simeq}\arrow{r}{w} &(\tau^{\leq i}Z)[1] \arrow{d}{\simeq}\\
	Y[n-1] & & X & Y[n]
\end{tikzcd}\]
splits or, equivalently, if and only if $w$ vanishes. As a confirmation, note that $w$ is
the morphism corresponding to the equivalence class of $\varepsilon$ under
the isomorphism $\Ext_\Acal^n(X,Y)\simeq\Hom_{\D(\Acal)}(X,Y[n])$.
\end{proof}

\subsection{CHZ-sequences}\label{sub:CHZ-sequences} We recall the main result of \cite{CHZ}, using the point of view expressed above. 

\begin{prop}[{\cite[Theorem A]{CHZ}}]\label{prop:CHZ}
Let $\mathcal A$ be an abelian category admitting a derived category $\D(\Acal)$, and let
$\mathbf{t}=(\mathcal T,\mathcal F)$ be a torsion pair in $\mathcal A$. The following are
equivalent:
\begin{enumerate}
\item $\mathbf{t}$ induces a bounded derived equivalence;
\item Every object $M$ of $\mathcal A$ admits an exact sequence of the form
	\begin{flalign*}
		\qquad\qquad\qquad\qquad\qquad \varepsilon_M\colon\qquad0\to F_M^0\to F_M^1\to M\to T_M^0\to T_M^1\to 0, &&
	\end{flalign*}
	with $T_M^0, T_M^1$ in $\mathcal T$ and $F_M^0, F_M^1$ in $\mathcal F$, which represents the zero
	element of $\Ext_{\mathcal A}^3(T_M^1,F_M^0)$; 
\item For every object $M$ of $\mathcal A$, there is a complex $Z_M:=F^1_M\to M \to T^0_M$
	with $T^0_M\in\Tcal$ in degree zero and $F^1_M\in\Fcal$, such that the truncation triangle
	\begin{flalign*}
		\qquad\qquad\qquad\qquad\qquad \Delta_M\colon\qquad\tau^{\leq -1}Z_M\longrightarrow Z_M\longrightarrow \tau^{\geq
		0}Z_M\longrightarrow (\tau^{\leq -1}Z_M)[1] &&
	\end{flalign*}
	splits, i.e. that $Z_M\cong \tau^{\leq -1}Z_M\oplus \tau^{\geq 0}Z_M$.
\end{enumerate}
Sequences as in (2) and complexes as in (3) will be called \textbf{CHZ-sequences} and \textbf{CHZ-complexes}.
\end{prop}
\begin{proof}
The proof that (1) is equivalent to (2) is essentially the cited result of \cite{CHZ}.  The proof
that (2) and (3) are equivalent is an immediate consequence of Lemma \ref{homological algebra} and of the fact that $\Fcal$ is closed under subobjects and $\Tcal$ is closed under quotient objects. 
\end{proof}

Note that, in some cases, we can avoid checking the vanishing of the Yoneda class of a candidate
CHZ-sequence. For example, if the abelian category $\mathcal A$ has global dimension less or equal
than two, the Yoneda bifunctor $\Ext^3_{\mathcal A}(-,-)$ is identically zero, and thus the
Yoneda-vanishing condition is automatic. Similar phenomena happen for certain torsion pairs. For
example, a complete analysis for cohereditary torsion pairs in Grothendieck categories can be
deduced from \cite[Corollary 4.1]{CHZ}. We will later focus on the class of hereditary torsion
pairs. Firstly, however, let us observe that the criterion given by the proposition above can be
simplified for Grothendieck abelian categories. 

\begin{prop}\label{CHZ generator}
Let $\Gcal$ be a Grothendieck abelian category with a generator $G$ and let
$\mathbf{t}=(\mathcal T,\mathcal F)$ be a torsion pair in $\Gcal$. Then $\mathbf{t}$ induces a
bounded derived equivalence if and only if there is an exact sequence
\begin{equation*}\begin{tikzcd}
0 \arrow{r} &F_G^0 \arrow{r}{a} &F_G^1 \arrow{r}{b} &G \arrow{r}{c} &T_G^0 \arrow{r}{d} &T_G^1 \arrow{r} &0
\end{tikzcd}\end{equation*}
with $T_G^0, T_G^1$ in $\mathcal T$ and $F_G^0, F_G^1$ in $\mathcal F$, which represents the zero
element of $\Ext_R^3(T_G^1,F_G^0)$.
\end{prop}
\begin{proof}
As in Proposition \ref{prop:CHZ}, we translate the existence of such a sequence to the existence of
a complex 
\[Z_G:=\quad\begin{tikzcd}F_G^1 \arrow{r}{b} &G \arrow{r}{c} &T_G^0\end{tikzcd}\]
with $T_G^0$ in degree zero for which the truncation triangle
($\Delta_G$) splits. First observe that since coproducts of split triangles are again split
and since both $\mathcal T$ and $\mathcal F$ are closed under coproducts, the existence of a
CHZ-complex is also guaranteed for any coproduct $M=G^{(I)}$, for any set $I$. In order to show that
such complexes exist for any $M$ in the category $\Gcal$, we may, therefore, suppose, without
loss of generality, that $M$ is a quotient of $G$ itself (otherwise, we could replace $G$ by
$G^{(I)}$ for a suitable set $I$).

Let then $f\colon G\to M$ be an epimorphism. Denote by $\overline{f}\colon \Coker(b)\to
\Coker(fb)$ the epimorphism induced by $f$. We define the following complex concentrated in degrees
$-2,-1$ and $0$,
\[Z_M:=\quad\begin{tikzcd}F_G^1 \arrow{r}{fb} &M \arrow{r}{\overline{c}} &T^0_G/K\end{tikzcd}\]
where $K=\Ker(\overline{f})$ and $\overline{c}$ is the composition of the projection $M\to
\Coker(fb)$ and the inclusion of $\Coker(fb)=\Coker(b)/K$ into $T^0_G/K$. Note that, by
construction, $\Ker(\overline{c})=\im(fb)$ and, thus, we have $H^{-1}(Z_M)=0$. Moreover, note that
$\Ker(fb)$ is a subobject of $F_G^1$ and that
\begin{equation}\label{eq:coker}\tag{$\ast$}
	\Coker(\overline{c})=(T^0_G/K)/(\Coker(b)/K)\cong T^0_G/\Coker(b)\cong T^1_G.
\end{equation}
This shows that $H^{-2}(Z_M)$ is torsionfree and $H^0(Z_M)$ is torsion. 
Finally, consider the map of complexes
\[\begin{tikzcd}
	Z_G\colon \arrow{d}[swap]{\varphi} &F_G^1 \arrow{r}{b}\arrow[equal]{d}
		&G \arrow{r}{c}\arrow{d}{f} &T_G^0 \arrow{d}{\pi}\\
	Z_M\colon &F_G^1 \arrow{r}{fb} &M \arrow{r}{\overline{c}} &T^0_G/K
\end{tikzcd}\]
where $\pi$ is the canonical projection. If we now consider the truncation triangles of $Z_G$ and
$Z_M$, as in Proposition \ref{prop:CHZ}, we get a morphism of triangles in $\D(\Gcal)$ as
follows:
\[\begin{tikzcd}
	F^0_G[2] \arrow{d}[swap]{\tau^{\leq -2}\varphi=H^{-2}(\varphi)[2]}\arrow{r}{a}
		&Z_G \arrow{rr}{d}\arrow{d}[swap]{\varphi} &
		&T^1_G \arrow{r}\arrow{d}[swap]{\tau^{\geq -1}\varphi=H^{0}(\varphi)}
		&F^0_G[3] \arrow{d}\\
	\Ker(fb)[2] \arrow{r}{\overline{a}} &Z_M \arrow{rr}{\overline{d}} &&T^1_G \arrow{r} &\Ker(fb)[3]
\end{tikzcd}\]
Finally, since $d$ is a split epimorphism and $H^0(\varphi)$ is an isomorphism (as seen in
\eqref{eq:coker}), we conclude that also
$\overline{d}$ is a split epimorphism, finishing our proof.
\end{proof}

\subsection{Hereditary torsion pairs inducing bounded derived equivalences}
We now focus on hereditary torsion pairs. The following lemma provides a simplification of the
criterion given in Proposition~\ref{prop:CHZ} for such torsion pairs.

\begin{lemma}\label{short CHZ}
Let $\mathbf{t}=(\mathcal T, \mathcal F)$ be a hereditary torsion pair in an abelian category
$\mathcal A$. If $\mathbf{t}$ is hereditary, then a CHZ-sequence for an object $M$ of $\mathcal A$
exists if and only if there exists a sequence of the form
\begin{equation*}\nonumber F_M\longrightarrow M\longrightarrow T_M\longrightarrow 0,\end{equation*}
with $T_M$ in $\mathcal T$ and $F_M$ in $\mathcal F$.
Such sequences will be called \textbf{short CHZ-sequences}.
\end{lemma}
\begin{proof}
If $\mathbf{t}$ is hereditary, any CHZ-sequence 
\begin{equation*}
	0\longrightarrow F_M^0\longrightarrow F_M^1\longrightarrow M\overset{b}{\longrightarrow} T_M^0\longrightarrow T_M^1\longrightarrow 0,
\end{equation*}
with $T_M^0, T_M^1$ in $\mathcal T$ and $F_M^0, F_M^1$ in $\mathcal F$,
 gives rise to a short
CHZ-sequence
\[F_M^1\longrightarrow M\longrightarrow \im(b)\longrightarrow 0,\]
where $\im(b)\leq T_M^0$ lies in $\Tcal$ because $\Tcal$ is closed under subobjects. Conversely,
since $\mathcal F$ is closed under subobjects, a short CHZ-sequence can be completed
with a torsion-free kernel on the left. Note that the class of the obtained exact sequence as a 3-extension is zero because
such a sequence is naturally a 2-extension; therefore, it is a CHZ-sequence.
\end{proof}

Given two objects $X$ and $M$ in a Grothendieck abelian category $\Gcal$, define the
\textbf{trace of $M$ in $X$} to be
\[\mathsf{tr}_M(X):=\sum\limits_{f\in \Hom_\Acal(M,X)}\im(f).\]
Notice that the trace of $M$ in $X$ is the biggest subobject of $X$ which is generated by $M$, in
the sense that it is a quotient of a coproduct of copies of $M$.

\begin{thm}\label{criterion Grothendieck}
Let $\Gcal$ be a Grothendieck abelian category with generator $G$, and let
$\mathbf{t}=(\Tcal,\Fcal)$ be a hereditary torsion pair in $\Gcal$ with torsion radical $t\colon
\Gcal\to \Tcal$. Then $\mathbf{t}$ induces a bounded derived equivalence if and only if
$G/\mathsf{tr}_{G/t(G)}(G)$ lies in $\Tcal$.
\end{thm}
\begin{proof}
If $G/\mathsf{tr}_{G/t(G)}(G)$ lies in $\Tcal$, let $I=\Hom_\Gcal(G/t(G),G)$ and $f\colon (G/t(G))^{(I)}\to G$
be the canonical morphism, whose image is precisely $\mathsf{tr}_{G/t(G)}(G)$. Then the sequence 
\[\begin{tikzcd}
	(G/t(G))^{(I)} \arrow{r}{f} &G \arrow{r} &G/\tr_{G/t(G)}(G) \arrow{r} &0
\end{tikzcd}\]
is a short CHZ-sequence for $G$ and, thus, by Lemma \ref{short CHZ} and Lemma \ref{CHZ generator},
we have that $\mathbf{t}$ induces a bounded derived equivalence. Conversely, if $\mathbf{t}$ induces a
bounded derived equivalence, then again by Lemma \ref{CHZ generator} and Lemma \ref{short CHZ} we have that
there is an exact sequence of the form
\[\begin{tikzcd}
	F_G \arrow{r}{h} &G \arrow{r} &T_G \arrow{r} &0
\end{tikzcd}\]
with $F_G$ in $\Fcal$ and $T_G$ in $\Tcal$.  Now, since by Lemma \ref{torsionfree generator} we know
that $F_G$
is a quotient of a coproduct of copies of $G/t(G)$, it follows that $\im(h)$ is contained in
$\mathsf{tr}_{G/t(G)}(G)$. Thus, $G/\mathsf{tr}_{G/t(G)}(G)$ is a quotient of $T_G$ and therefore it
lies in $\Tcal$, as wanted.
\end{proof}

\begin{cor}
Let $\Gcal$ be a Grothendieck category and $\mathbf{t}_1:=(\Tcal_1,\Fcal_1)$, $\mathbf{t}_2:=(\Tcal_2,\Fcal_2)$ be hereditary torsion
pairs in $\Gcal$ inducing bounded derived equivalence. Then their meet induces a bounded derived equivalence.
\end{cor}
\begin{proof}
Let $\mathbf{t}_3:=(\Tcal_3:=\Tcal_1\cap\Tcal_2,\Fcal_3)$ be the meet of $\mathbf{t}_1$ and $\mathbf{t}_2$.
It is again hereditary, so we can apply Theorem
\ref{criterion Grothendieck}.  Let $G$ be a generator of $\Gcal$, and $t_i(G)$ the torsion
part of $G$ with respect to $\mathbf{t}_i$, for $i=1,2,3$.

Since $t_3(G)$ lies in $\mathcal T_3\subseteq\mathcal T_1$, it is a subobject of $t_1(G)$, so that
$G/t_3(G)\twoheadrightarrow G/t_1(G)$. It follows that
$\tr_{G/t_1(G)}(G)\hookrightarrow\tr_{G/t_3(G)}(G)$, and $G/\tr_{G/t_1(G)}(G) \twoheadrightarrow
G/\tr_{G/t_3(G)}(G)$.  Now the source of this epimorphism is in $\mathcal T_1$, so the target is as
well; a similar argument for $t_2(G)$ and $t_3(G)$ shows then that
$G/\tr_{G/t_3(G)}(G)$ belongs to $\mathcal T_1\cap\mathcal T_2=\mathcal T_3$.
\end{proof}

Let us now restrict our setting to the Grothendieck abelian category $\Mod(R)$ of right $R$-modules over a
ring $R$. Recall that given a right $R$-module $M$, its
\textbf{(right) annihilator} is the ideal of $R$ 
\[\mathsf{Ann}(M_R):=\{r\in R:Mr=0\}.\]
Similarly, we define the annihilator of a left module ${}_RM$. Notice, however, that the annihilator of a
module, be it right or left, is always a two-sided ideal. 

\begin{prop}
Let $R$ be a ring and $\mathbf{t}=(\Tcal,\Fcal)$ a torsion pair in $\mathsf{Mod}(R)$. Let
$K_\mathbf{t}$ be the two-sided ideal obtained as the intersection of all the annihilators of
modules in $\Fcal$. The following statements hold:
\begin{enumerate}
\item $R/K_\mathbf{t}$ lies in $\mathcal F$;
\item $\Fcal$ is a subcategory of $\Mod(R/K_\mathbf{t})$, where
	$\Mod(R/K_\mathbf{t})$ is identified with the subcategory of $R$-modules annihilated by
	$K_\mathbf{t}$;
\item $t(R)$ coincides with $K_\mathbf{t}$.
\end{enumerate}
\end{prop}
\begin{proof}
(1) For every element $r$ in $R\setminus K_\mathbf{t}$, there is an object $F_r$ in $\mathcal F$ and an element
$f_r$ in $F_r$ such that $f_rr\neq 0$. Consider the torsionfree module $F:=\prod_{r\in R\setminus
K} F_r$, and the morphism $R\to F$ defined by $1\mapsto (f_r)_r$. Its kernel is clearly
$K_\mathbf{t}$, so it induces an embedding (of right $R$-modules) $R/K_\mathbf{t}\hookrightarrow F$. Hence
$R/K_\mathbf{t}$ lies in $\mathcal F$.

(2) The claim follows from the fact that $K_\mathbf{t}$ annihilates every module in $\Fcal$ by
definition.

(3) Consider the diagram:
\[\begin{tikzcd}
	0 \arrow{r} &t(R) \arrow{r} & R \arrow[equal]{d}\arrow{r}{p_1}
		&R/t(R) \arrow[dotted, shift right, swap]{d}{c}\arrow{r} &0\\
	0 \arrow{r} &K_\mathbf{t} \arrow{r} &R \arrow{r}{p_2}
		&R/K_\mathbf{t} \arrow[dotted, shift right, swap]{u}{d}\arrow{r} &0
\end{tikzcd}\]
Since $R/t(R)$ lies in $\mathcal F$, which is contained in $\Mod(R/K_\mathbf{t})$ by (2), the
map $p_1$ factors through $R/K_\mathbf{t}$, yielding the dotted epimorphism $d$. On the other
hand, $R/K_\mathbf{t}$ is torsionfree (by (1)), and so the map $p_2$ factors through $R/t(R)$,
giving the dotted epimorphism $c$. Finally it is an easy verification that $c$ and $d$ are inverse
to each other, showing that indeed $K_\mathbf{t}=t(R)$ as ideals of $R$.
\end{proof}

\begin{rmk}
Combining items (2) and (3) of the previous proposition one recovers Lemma \ref{torsionfree generator}.
\end{rmk}

We can finally apply Theorem \ref{criterion Grothendieck} to the case of a module category.

\begin{thm}\label{thm:her-criterion}
Let $R$ be a ring and let $\mathbf{t}=(\mathcal T,\mathcal F)$ be a hereditary torsion pair in
$\Mod(R)$. Then $\mathbf{t}$ induces a bounded derived equivalence if and only if $R/\Ann({}_Rt(R))$
lies in $\Tcal$. This holds in particular whenever ${}_Rt(R)$ is finitely generated (as a left
$R$-module) and $\Ann({}_Rt(R))=\Ann(t(R)_R)$.
\end{thm}
\begin{proof}
The first assertion follows immediately from Theorem \ref{criterion Grothendieck} once we observe that
$\tr_{R/t(R)}(R)=\Ann({}_Rt(R))$.
Indeed, by definition an element $r$ of $R$ lies in $\tr_{R/t(R)}(R)$ if and only if it lies in the image of a
morphism $f\colon R/t(R)\to R$: and this happens if and only if $rt(R)=0$, i.e. if
$r$ lies in $\Ann({}_Rt(R))$.

For the second statement, let $x_1,...,x_n$ be elements in $R$ such that
${}_Rt(R)=\sum_{k=1}^nRx_k$, and consider the morphism of right $R$-modules
$\alpha\colon R\to t(R)^{\oplus n}$ defined by $1\mapsto (x_1,\dots, x_n)$. Then its kernel is
$\Ann(t(R)_R)$. Indeed, if $r$ is in $\Ann(t(R)_R)$ then $x_kr=0$ for all $1\leq k\leq n$, and so
$\alpha(r)=(x_1r,\dots,x_nr)=0$. Conversely, if $r\in\ker \alpha$, then $x_kr=0$ for all $1\leq
k\leq n$. Now every $x$ in $t(R)$ can be written as $\sum_{k=1}^n r_kx_k$, so $xr=\sum_{k=1}^n
(r_kx_k)r=0$. By our hypothesis, then, we also have $\ker \alpha=\Ann({}_Rt(R))$. Therefore,
$\alpha$ induces a monomorphism of right $R$-modules
\[R/\Ann({}_Rt(R))=R/\Ann(t(R)_R)\hookrightarrow t(R)^{\oplus n}.\]
Since the latter is a torsion module and $\mathbf{t}$ is hereditary, we conclude that
$R/\Ann({}_Rt(R))$ is torsion.
\end{proof}

\begin{cor}\label{der eq comm ring}
Let $R$ be a noetherian ring and suppose it is either commutative or semiprime. Then every
hereditary torsion pair induces an unbounded derived equivalence.
\end{cor}
\begin{proof}
The fact that a hereditary torsion pair $\mathbf{t}$ of $\Mod(R)$ induces bounded derived equivalence follows immediately from
Theorem \ref{thm:her-criterion}; indeed, for commutative or semiprime rings, left and right
annihilators of ideals coincide and, furthermore, since $R$ is left noetherian, ${}_Rt(R)$ is
finitely generated.

Now, the $t$-structure obtained by HRS-tilting $\Mod(R)$ at $\mathbf{t}$ is smashing with
Grothendieck heart by Remark \ref{rmk:here-finitetype} and item (1) of Theorem \ref{Saorin};
therefore by applying item (2) of Proposition \ref{prop:bound2unbound} we get that $\mathbf{t}$
induces an unbounded derived equivalence.
\end{proof}

The corollary above has a direct implication in silting theory for commutative noetherian rings.

\begin{cor}\label{2-term cosilting is cotilting}
Every two-term cosilting complex over a commutative noetherian ring is cotilting.
\end{cor}
\begin{proof}
If $R$ is a commutative noetherian ring, then the HRS-tilting $t$-structure at any hereditary torsion pair in $\mathsf{Mod}(R)$ is a cosilting $t$-structure associated with a two-term cosilting complex (\cite[Corollary 4.1, Lemma 4.2]{AH}). Since, by Corollary \ref{der eq comm ring}, this $t$-structure induces a derived equivalence, the two-term cosilting complex must be cotilting (\cite[Corollary 5.2]{PV}).
\end{proof}

\section{$t$-structures inducing derived equivalences for commutative noetherian rings}

We turn now our attention back to compactly generated $t$-structures in the derived category of a
commutative noetherian ring $R$. In particular, we aim to find sufficient conditions for a given
intermediate compactly generated $t$-structure to induce a derived equivalence. 

\begin{rmk}\label{rmk:comm-noeth-unbounded}

Recall that by Theorem \ref{thm:$t$-structure-types}, intermediate compactly generated $t$-structures
in $\D(R)$  satisfy the hypotheses of item (2) of Proposition
\ref{prop:bound2unbound}. Therefore, their heart is derived equivalent to $\Mod(R)$ if and only if
it is bounded derived equivalent. This includes the case of the HRS-tilt of an intermediate
compactly generated $t$-structure at a torsion pair of finite type (see item (1) of Theorem
\ref{Saorin} and Theorem \ref{t-st comm}).
In the following we will use this fact without an explicit
mention.
\end{rmk}

\subsection{Sufficient conditions for derived equivalence}
Let $R$ be a commutative noetherian ring, and consider a hereditary torsion pair $\mathbf{t}$ in
$\Mod(R)$. If $\Tbb_\mathbf{t}$ denotes the $t$-structure obtained by HRS-tilting
$\Mod(R)$ at $\mathbf{t}$, with heart $\Hcal_\mathbf{t}$, we know that:
\begin{itemize}
\item $\Tbb_\mathbf{t}$ is an intermediate compactly generated $t$-structure (Remark
\ref{rmk:hrs-at-hereditary});
\item $\Hcal_\mathbf{t}$ is a locally coherent Grothendieck category (Theorem \ref{Saorin}, Remark \ref{rmk:here-finitetype} and
Theorem \ref{review torsion}(3));
\item $\Hcal_\mathbf{t}$ is derived equivalent to $\Mod(R)$ (Corollary \ref{der eq comm ring}).
\end{itemize}

This subsection takles the question of whether we can proceed with a chain of HRS-tilts at suitable
torsion pairs, so that all the obtained $t$-structures will retain these properties. For this purpose, we need to first extend the scope of Theorem \ref{criterion Grothendieck} in order to apply it not only to HRS-tilts of the heart of the standard t-structure but also to hearts of t-structures inducing derived equivalences.

\begin{lemma}\label{adaptation} Let $R$ be a ring and $\mathcal A$ the heart of an intermediate t-structure $\mathbb{T}=(\Ucal,\Vcal)$ in $\mathsf{D}(R)$ such that $\mathbb{T}$ induces a derived equivalence. Given a torsion pair $\mathbf{t}$ in $\mathcal{A}$, let $\mathbb{S}$ be the t-structure in $D(\mathcal{A})$ produced by HRS-tilting the standard t-structure in $D(\mathcal{A})$ with respect to $\mathbf{t}$, and let $\mathbb{T}_\mathbf{t}$ be the t-structure in $\mathsf{D}(R)$ obtained by HRS-tilting the t-structure $\mathbb{T}$ in $\mathsf{D}(R)$. Then $\mathbb{S}$ induces a derived equivalence if and only if so does $\mathbb{T}_\mathbf{t}$.
\end{lemma}
\begin{proof}
Let $\rho_1\colon \mathsf{D}^b(\mathcal{A})\longrightarrow \mathsf{D}^b(R)$ be a realisation functor for the t-structure $\mathbb{T}$ (which is assumed to be a triangle equivalence). It can easily be checked, using elementary properties of realisation functors, that the t-structure $\mathbb{S}\cap \mathsf{D}^b(\mathcal A)$ in $\mathsf{D}^b(\mathcal A)$ is sent, via the equivalence $\rho_1$, to the t-structure $\mathbb{T}_\mathbf{t}\cap \mathsf{D}^b(R)$ in $\mathsf{D}^b(R)$. In particular, $\rho_1$ induces an equivalence between $\mathcal{H}_\mathbb{S}$  and $\mathcal{A}_\mathbf{t}$, the heart of $\mathbb{T}_\mathbf{t}$ in $\mathsf{D}(R)$. 

Suppose now that $\mathbb{S}$ induces a derived equivalence, i.e. the realisation functor $\rho_2\colon \mathsf{D}^b(\Hcal_\mathbb{S})\longrightarrow \mathsf{D}^b(\Acal)$ is a triangle equivalence.   Hence, it follows that for $X$ and $Y$ in $\Acal_\mathbf{t}$, with $X\cong\rho_1X'$ and $Y=\rho_1Y'$ for $X'$ and $Y'$ in $\mathcal{H}_\mathbb{S}$, we have $\mathsf{Ext}^n_{\Acal_\mathbf{t}}(X,Y)\cong \mathsf{Ext}^n_{\mathcal{H}_\mathbb{S}}(X',Y')$ and, since $\rho_2$ is a derived equivalence, the latter is isomorphic to $\mathsf{Hom}_{\mathsf{D}^b(\Acal)}(X',Y'[n])$. Hence, we have
$$\mathsf{Ext}^n_{\Acal_\mathbf{t}}(X,Y)\cong \mathsf{Hom}_{\mathsf{D}^b(\Acal)}(X',Y'[n])\cong \mathsf{Hom}_{\mathsf{D}^b(\Acal)}(\rho_1X',\rho_1Y'[n])\cong \mathsf{Hom}_{\mathsf{D}^b(R)}(X,Y[n]),$$
proving that $\mathbb{T}_\mathbf{t}$ induces a derived equivalence. The converse implication is completely analogous.
\end{proof}

As a consequence of the lemma above, we can use the criteria in Proposition \ref{prop:CHZ} and, consequently, in Theorem \ref{criterion Grothendieck} to test whether HRS-tilts of hearts of t-structures that induce derived equivalences still induce derived equivalences.

\begin{thm}\label{main der eq}
Let $\mathbb{T}=(\Ucal,\Vcal)$ be an intermediate compactly generated $t$-structure in $\D(R)$ with
heart $\Hcal_\Tbb$, and suppose that $\mathbb{T}$ induces a derived equivalence.
Let $\mathbf{t}=(\Tcal_V,\Fcal_V)$ be the hereditary torsion pair of finite type
in $\Hcal_\Tbb$ associated
to a specialisation closed subset $V$, and let $t\colon \Hcal_\Tbb\to \Tcal$ denote the corresponding
torsion radical. If there is a set of generators $\{G_\lambda: \lambda\in \Lambda\}$ of $\Hcal_\Tbb$ such
that $G_\lambda/t(G_\lambda)$ is finitely presented, then the HRS-tilted $t$-structure associated to
$\mathbf{t}$ induces a derived equivalence. 
\end{thm}
\begin{proof}
We use an adaptation of Theorem
\ref{criterion Grothendieck} (see also Lemma \ref{adaptation}). In fact, we will check that for each $\lambda$ in $\Lambda$, the
quotient $G_\lambda/\mathsf{tr}_{G_\lambda/t(G_\lambda)}(G_\lambda)$ is torsion; this provides us
with a sequence such as in Proposition \ref{CHZ generator} by considering the coproduct over the set
$\Lambda$ of the resulting sequences. This will show that, indeed, the HRS-tilted $t$-structure
associated to $\mathbf{t}$ induces a derived equivalence (see Remark \ref{rmk:comm-noeth-unbounded}).

First observe that since $V$ is specialisation closed and $\Hcal_\Tbb$ is contained in $D^b(R)$ ($\Tbb$ is intermediate), we have from Equation (\ref{small vs big}) that 
$\supp^{-1}(V)\cap\Hcal_\Tbb=\Supp^{-1}(V)\cap\Hcal_\Tbb$ and, therefore,
all we want to show is that $G_\lambda/\mathsf{tr}_{G_\lambda/t(G_\lambda)}(G_\lambda)\otimes_R
R_\mathfrak{p}=0$ for all primes $\mathfrak{p}$ that do not lie in $V$.
Let then $\mathfrak{p}$ be a
prime not in $V$. Since $R_\mathfrak{p}$ is flat, we can write it as a direct limit of free
$R$-modules, i.e. $R_\mathfrak{p}=\varinjlim_{j\in J} R^{(n_j)}$.
Since directed homotopy colimits in $\D(R)$ are computed as componentwise direct limits, one sees
that
\[G_\lambda\otimes R_\mathfrak{p}=G_\lambda\otimes \varinjlim\nolimits_J
R^{(n_j)}=\Hocolim\nolimits_J G_\lambda^{(n_j)}\]
(see the proof of \cite[Lemma 4.1]{H} for the details).
Now, by \cite[Corollary 5.8]{SSV}, this directed homotopy colimit of objects of $\Hcal_\Tbb$ is a
direct limit in $\Hcal_\Tbb$.
We can therefore use the hypothesis that $G_\lambda/t(G_\lambda)$ is finitely presented in $\Hcal_\Tbb$ to
write
\[\Hom_{\D(R)}(G_\lambda/t(G_\lambda),G_\lambda)\otimes_R
	R_\mathfrak{p}\cong \varinjlim\nolimits_J \Hom_{\D(R)}(G_\lambda/t(G_\lambda),G_\lambda)^{(n_j)} \cong
	\Hom_{\D(R)}(G_\lambda/t(G_\lambda),G_\lambda\otimes_R R_\mathfrak{p}).\]
Since $\D(R_\mathfrak{p})$
is a bireflective subcategory of $\D(R)$ with reflection functor
$-\otimes_RR_\mathfrak{p}$, it follows that
$$\Hom_{\D(R)}(G_\lambda/t(G_\lambda),G_\lambda)\otimes_R R_\mathfrak{p}\cong
\Hom_{\D(R)}(G_\lambda/t(G_\lambda)\otimes_RR_\mathfrak{p},G_\lambda\otimes_R R_\mathfrak{p}).$$ Now,
observe that since $\mathfrak{p}$ does not lie in $V$ and since $\Tcal_V$ is supported in $V$, we
have $t(G_\lambda)\otimes_RR_\mathfrak{p}=0$. Since $-\otimes_RR_\mathfrak{p}$ is an exact functor
in $\Hcal_{\Tbb}$, it then follows that $G_\lambda\otimes_RR_\mathfrak{p}\cong
G_\lambda/t(G_\lambda)\otimes_RR_\mathfrak{p}$. Therefore, the isomorphism of $\Hom$-spaces in the
last equation above allows us to conclude that there is a morphism $f\colon G_\lambda/t(G_\lambda)\to
G_\lambda$ such that $f\otimes_RR_\mathfrak{p}$ is an isomorphism. Since
$\mathsf{tr}_{G_\lambda/t(G_\lambda)}(G_\lambda)\otimes_R R_\mathfrak{p}$ contains
$\mathsf{Im}(f)\otimes_R R_\mathfrak{p}$, we have that
$\mathsf{tr}_{G_\lambda/t(G_\lambda)}(G_\lambda)\otimes_R R_\mathfrak{p}=G_\lambda\otimes_R
R_\mathfrak{p}$. We thus conclude that, indeed,
$G_\lambda/\mathsf{tr}_{G_\lambda/t(G_\lambda)}(G_\lambda)\otimes_R R_\mathfrak{p}=0$, as wanted.
\end{proof}
\begin{cor}\label{loc coh}
Let $\mathbb{T}=(\Ucal,\Vcal)$ be an intermediate compactly generated $t$-structure in $\D(R)$ with
a locally coherent heart $\Hcal_{\Tbb}$, and suppose that $\mathbb{T}$ induces a derived equivalence. Let
$\mathbf{t}=(\Tcal_V,\Fcal_V)$ be the hereditary torsion pair of finite type in $\Hcal_{\Tbb}$ associated
to a specialisation-closed $V$ and suppose that $(\Tcal_V,\Fcal_V)$ is restrictable.
Then the HRS-tilted $t$-structure associated to $\mathbf{t}$ induces a 
derived equivalence. 
\end{cor}
\begin{proof}
Under the assumption that $\mathbf{t}$ is restrictable, for any set
$\{G_\lambda\colon\lambda\in\Lambda\}$ of finitely presented generators of $\Hcal_{\Tbb}$,
$G_\lambda/t(G_\lambda)$ will also be finitely
presented. Hence, the result follows from Theorem \ref{main der eq}.
\end{proof}

\begin{rmk}
Since we know that for a commutative noetherian ring $R$, every hereditary torsion pair in
$\mathsf{Mod}(R)$ is restrictable, note that Corollary \ref{der eq comm ring} follows immediately from the
corollary above.
\end{rmk}
\begin{rmk}
If $\Tbb$ is as in Corollary \ref{loc coh}, $\mathbf{t}$ is any torsion pair in $\Hcal_{\Tbb}$  and we
happen to know that $\Hcal_\mathbf{t}$ is locally coherent, then the torsion pair $\mathbf{t}$ is
restrictable if and only if there is a set $\{G_\lambda\colon\lambda\in\Lambda\}$ of finitely presented
generators of $\Hcal_{\Tbb}$ such that $G_\lambda/t(G_\lambda)$ is finitely presented (see \cite[Remark
6.3(3)]{PSV}). Therefore, knowing this information about $\Hcal_\mathbf{t}$, the hypothesis of
Corollary \ref{loc coh} is minimal to apply Theorem \ref{main der eq}.
\end{rmk}

\subsection{Intermediate compactly generated $t$-structures via iterated HRS-tilts}

In this subsection, we show that any intermediate compactly generated $t$-structure of $\D(R)$ is
obtained from the standard $t$-structure by a sequence of HRS-tilts at hereditary torsion pairs of
finite type. The following proposition gives us some TTF classes in hearts of compactly generated t-structure (compare with Remark \ref{TTF in HRS}).

\begin{prop}\label{T0}
Let $\varphi$ be an intermediate sp-filtration, with $\varphi(0)\neq\varphi(1)=\emptyset$, and denote by
$\Hcal_\varphi$ the heart of the associated compactly generated $t$-structure
$\mathbb{T}_\varphi=(\Ucal_\varphi,\Vcal_\varphi)$. Then $\Tcal_0:=\Hcal_\varphi\cap
\mathsf{Mod}(R)$ is a TTF class in $\Hcal_\varphi$. In particular, we have that
$\Tcal_0=\supp^{-1}(\varphi(0))\cap\Hcal_\varphi=\mathsf{Supp}^{-1}(\varphi(0))\cap\Hcal_\varphi$.
\end{prop}
\begin{proof}
First note that
\[\Hcal_\varphi\cap\Dbb^{\geq 0}\overset{(1)}{=} \Tcal_0
\overset{(2)}{=} \Ucal_\varphi\cap\Mod(R) \overset{(3)}{=} \Supp^{-1}(\varphi(0))\cap\Mod(R).\]
Indeed, equality (1) follows from $\Hcal_\varphi\subseteq\Ucal_\varphi\subseteq\Dbb^{\leq0}$, (2)
follows from $\Mod(R)\subseteq\Dbb^{\geq 0}\subseteq\Vcal_\varphi[1]$, and (3) follows by
definition of $\Ucal_\varphi$. In particular, (3) shows that
$\supp(\Tcal_0)=\Supp(\Tcal_0)=\varphi(0)$ by Theorem \ref{review torsion}.

We now show that $(\Hcal_\varphi\cap\Dbb^{\leq -1},\Tcal_0)=(\Hcal_\varphi\cap\Dbb^{\leq -1},\Hcal_\varphi\cap\Dbb^{\geq
0})$ is a
torsion pair in $\Hcal_\varphi$.
First, $\Hom_{\Hcal_\varphi}(\Hcal_\varphi\cap\Dbb^{\leq-1},\Tcal_0)=0$ is clear. Now, let $X$ be
an object of $\Hcal_\varphi$ and consider the truncation triangle with respect to the standard
$t$-structure
\[\tau^{\leq-1}X\longrightarrow X\longrightarrow H^0(X)\longrightarrow (\tau^{\leq-1}X)[1].\]
By definition of
$\Ucal_\varphi$, it is closed under standard truncations, so all the vertices belong to
$\Ucal_\varphi$. Moreover, it is also clear that $H^0(X)$ lies in
$\Dbb^{\geq0}\subseteq\Vcal_\varphi[1]$, and therefore $H^0(X)$ lies in $\Hcal_\varphi$. Lastly,
since $\Vcal_\varphi[1]$ is closed under taking co-cones, also $\tau^{\leq-1}X$ belongs to
$\Vcal_\varphi[1]$, and hence $\tau^{\leq-1}X$ lies also in $\Hcal_\varphi$. The triangle above is then
a short exact sequence in $\Hcal_\varphi$ and it is the torsion decomposition of $X$ with respect to
the torsion pair $(\Hcal_\varphi\cap\Dbb^{\leq -1},\Tcal_0)$.

It remains to show that $\Tcal_0$ is also a torsion class in $\Hcal_\varphi$ or, equivalently, that
$\Tcal_0$ is closed under quotients in $\Hcal_\varphi$. Consider a short exact sequence in
$\Hcal_\varphi$
\[0\longrightarrow X\longrightarrow Y\longrightarrow Z\longrightarrow 0\]
where $Y$ lies in $\Tcal_0$. Since $\Tcal_0$ is a
torsionfree class in $\Hcal_\varphi$, it follows that $X$ is also in $\Tcal_0$. Since
$\mathsf{supp}(Z)\subseteq \mathsf{supp}(X[1])\cup\mathsf{supp}(Y)$, it follows that
$\mathsf{supp}(Z)\subseteq \varphi(0)$. It remains to show that $Z$ lies in $\mathsf{Mod}(R)$.
Applying the standard cohomology functor to the triangle induced by the short exact sequence above,
we observe that, since $\Tcal_0$ is a hereditary torsion class in $\mathsf{Mod}(R)$, $H^{-1}(Z)$
lies in $\Tcal_0\subseteq\Hcal_\varphi$. Moreover, the above paragraph has also shown that
$\tau^{\leq-1}Z\cong H^{-1}(Z)[1]$ lies in $\Hcal_\varphi$. But this means that $H^{-1}(Z)=0$ and $Z$
must then lie in $\Tcal_0$.

Finally, the last statement follows from the fact that $\varphi(0)$ is specialisation closed,
$\Hcal_\varphi$ is contained in $\D^b(R)$ and hereditary torsion classes in $\Hcal_\varphi$ are determined by their support (see Proposition
\ref{some properties}). 
\end{proof}

\begin{lemma}\label{tilt between cg}
Let $\varphi$ and $\psi$ be intermediate sp-filtrations such that 
$\varphi(1)=\psi(1)=\emptyset$ and such that $\psi(i)=\varphi(i+1)$ for every
$i<0$. Then the compactly generated $t$-structure $\mathbb{T}_\psi$ associated to $\psi$ is obtained by HRS-tilting $\mathbb{T}_\varphi=(\Ucal_\varphi,\Vcal_\varphi)$ (with heart $\Hcal_\varphi$) with respect to a hereditary torsion pair of finite type whose torsion class is $\Supp^{-1}(\psi(0))\cap\Hcal_\varphi$.
\end{lemma}
\begin{proof}
Let $\Tcal_i$ denote the hereditary torsion class in $\Mod(R)$ supported on $\varphi(i)$, for any integer $i$.
Since $\psi(0)\subseteq\psi(-1)=\varphi(0)$, we have that
$\Tcal':=\Hcal_\varphi\cap\Supp^{-1}(\psi(0))\subseteq\Hcal_\varphi\cap\Supp^{-1}(\varphi(0))=\Tcal_0$,
the last equality following from Proposition \ref{T0}. We know from Theorem \ref{classification of tp in heart}
that $\Tcal'$ is a hereditary torsion class in $\Hcal_\varphi$.  If we tilt
$\Hcal_\varphi$ with respect to $\Tcal'$ we obtain a $t$-structure having aisle
\[\overline{\Ucal}:= \Ucal_\varphi[1]\ast\Tcal'.\]
Now, since we have that $\Ucal_\varphi[1]\subseteq\Dbb^{\leq-1}$ and that 
$\Tcal'\subseteq\Tcal_0\subseteq\Dbb^{\geq0}$, this aisle $\overline{\Ucal}$ consists of the objects $X$ such that the standard truncation
$\tau^{\leq-1}X$ lies in $\Ucal_\varphi[1]$ and the standard truncation $\tau^{\geq0}X$ lies in $\Tcal'$, i.e.
\[\overline{\Ucal}=\{X\in \D(R)\colon \Supp(H^iX) \subseteq\varphi(i+1)=\psi(i), \ \forall i<0,\ \Supp(H^0X)\subseteq\psi(0)\}\]
In other words, we have that $\overline{\Ucal}$ is the aisle of the $t$-structure determined by
$\psi$. Moreover, since this is also a compactly generated $t$-structure, it follows that the
hereditary torsion pair we have tilted at is of finite type (see item (1) of Theorem \ref{Saorin}).
\end{proof}

\begin{ntn}
Note that in the above lemma, the sp-filtration $\varphi$ can be recovered from $\psi$. We will
denote this operation on sp-filtrations by writing $\varphi=\psi^{\langle
1\rangle}$. In the notation of \cite[\S 5.3]{AJS}, we have that $\psi^{\langle
1\rangle}$ is a shift of $\psi'$, i.e. $\psi^{\langle 1\rangle}(i)=\psi'(i-1)$.
Moreover, starting with an sp-filtration $\psi$ such that $\psi(1)=\emptyset$,
we will denote the iterations of this process by $\psi^{\langle n\rangle}$, for
$n\geq 1$:
\[\psi^{\langle n\rangle}(i) = \begin{cases}
		\emptyset &\text{if }i>0\\
		\psi(i-n) &\text{if }i\leq 0.
	\end{cases}\]
\end{ntn}

\begin{prop}\label{iterated HRS}
Let $\varphi$ be an intermediate sp-filtration such that $\varphi(1)=\emptyset$.
Then the compactly generated $t$-structure $\mathbb{T}_\varphi$ associated to $\varphi$ can be built
from the standard $t$-structure by an iteration of HRS-tilts at hereditary torsion pairs of finite
type having specialisation-closed support.
\end{prop}
\begin{proof}
Since $\varphi$ is intermediate, we have $\Spec(R)=\varphi(-n)\supsetneq \varphi(-n+1)$ for some $n\geq0$.
The statement then follows by induction on $n$, using Lemma \ref{tilt between cg}.
\end{proof}

\subsection{Restrictable $t$-structures and derived equivalences}

We now turn to restrictable, intermediate and compactly generated $t$-structures, with the aim of
establishing that they induce derived equivalences. We begin by reviewing what is known about how to
characterise the sp-filtrations associated to the restrictable compactly generated $t$-structures
(see \cite{AJS}). The following condition turns out to play a significant role in that
characterisation for some commutative rings.

\begin{dfn}
An sp-filtration $\varphi$ is said to satisfy the \textbf{weak Cousin condition} if whenever
$\mathfrak{p}$ and $\mathfrak{q}$ are prime ideals such that $\mathfrak{p}\subsetneq \mathfrak{q}$
and $\mathfrak{p}$ is maximal under $\mathfrak{q}$ (i.e. there is no prime ideal $\mathfrak{t}$ such
that $\mathfrak{p}\subsetneq \mathfrak{t}\subsetneq \mathfrak{q}$), then we have 
\[\forall j\in\mathbb{Z},\ \mathfrak{q}\in\varphi(j)\Rightarrow \mathfrak{p}\in\varphi(j-1)\]
\end{dfn}

\begin{thm}\cite[Theorem 3.10 and 6.9, Corollary 4.5]{AJS}\cite[Theorem 6.3]{S}\label{thm restrictable}
Let $R$ be a commutative noetherian ring, $\mathscr{B}$ the set of $t$-structures in
$\D^b(\mathsf{mod}(R))$ and $\mathscr{T}$ the set of compactly generated $t$-structures in
$\D(R)$. There is an assignment $\Theta\colon \mathscr{B}\longrightarrow
\mathscr{T}$, sending a $t$-structure $\mathbb{B}:=(\Xcal,\Ycal)$ in $\D^b(\mathsf{mod}(R))$ to the
$t$-structure generated by $\Xcal$, namely
$\Theta(\mathbb{B}):=({}^\perp(\Xcal^\perp),\Xcal^\perp)$. Moreover, for every $\mathbb{B}$ in
$\mathscr{B}$, we have
\begin{enumerate}
\item $\Theta(\mathbb{B})\cap \D^b(\mathsf{mod}(R))=\mathbb{B}$ (and, in particular, $\Theta$ is injective);
\item The sp-filtration associated to $\Theta(\mathbb{B})$ satisfies the weak Cousin condition;
\item The heart of $\Theta(\mathbb{B})$ is locally coherent and its subcategory of finitely presented objects coincides with the heart of $\mathbb{B}$.
\end{enumerate}
The image of $\Theta$ is, then, the set of restrictable compactly generated $t$-structures.
Moreover, if $R$ admits a dualising complex, then the $t$-structures in the image of $\Theta$ are
those whose associated sp-filtrations satisfy the weak Cousin condition.
\end{thm}

\begin{dfn}
Let $R$ be a commutative noetherian ring. We say that an sp-filtration $\varphi$ in $\Spec(R)$ is
\textbf{restrictable} if the associated compactly generated $t$-structure is
restrictable (in other words, the associated $t$-structure is in the image of the assignment
$\Theta$).
\end{dfn}

Note that it follows easily from the definition that if $\varphi$ is an sp-filtration with
$\varphi(1)=\emptyset$ satisfying the weak Cousin condition, then $\varphi^{\langle n\rangle}$ also
satisfies the weak Cousin condition, for every $n\geq 1$. In fact, the following related statement holds.

\begin{prop}\cite[Lemma 5.7]{AJS}\label{restrictable iteration}
If an intermediate sp-filtration $\varphi$ is restrictable, then $\varphi^{\langle n\rangle}$ is
also (intermediate and) restrictable, for any $n\geq1$.
\end{prop}

\begin{cor}\label{seq restrictable torsion pairs}
Let $\varphi$ be an intermediate sp-filtration with $\varphi(1)=\emptyset$. If $\varphi$ is
restrictable, then the torsion pairs involved in the HRS-tilts of Proposition \ref{iterated HRS} are
restrictable.
\end{cor}
\begin{proof}
By Proposition \ref{iterated HRS}, we know that $\mathbb{T}_\varphi$ is obtained from the standard
$t$-structure by iterated HRS-tilts at hereditary torsion pairs of finite type. Moreover, by
Proposition \ref{restrictable iteration}, each of the $t$-structures involved in this process is
restrictable and, in particular, their hearts are locally coherent with finitely presented
objects given by the intersection with $\D^b(\mathsf{mod}(R))$, by Theorem \ref{thm restrictable}.
Finally, Theorem \ref{Saorin} completes the argument.
\end{proof}

\begin{thm}\label{der eq restrictable}
Let $R$ be a commutative noetherian ring and let $\mathbb{T}$ be an intermediate
restrictable compactly generated $t$-structure in $\D(R)$. Then $\mathbb{T}$ induces a derived equivalence.
\end{thm}
\begin{proof}
Up to shifting $\Tbb$, we may assume that the the associated sp-filtration $\varphi$ has
$\varphi(1)=\emptyset$. Proposition~\ref{iterated HRS} then shows that $\Tbb$ can be obtained from
the standard $t$-structure via an
iterated HRS-tilting process involving hereditary torsion pairs of finite type. Moreover, Corollary \ref{seq restrictable
torsion pairs} shows that these torsion pairs are restrictable. Applying Corollary \ref{loc coh} to these tilts, we obtain a chain of triangle equivalences
linking $\D(\Hcal_\Tbb)$ and $\D(R)$, as wanted.
\end{proof}

\begin{cor}\label{bdd cosilting is cotilting}
Let $R$ be a commutative noetherian ring. Every bounded cosilting object of $\D(R)$ whose $t$-structure is restrictable is cotilting.
\end{cor}
\begin{proof}
	It follows from \cite[Proposition 3.10]{MV} that every bounded cosilting
object is pure-injective. The associated
$t$-structure is then compactly generated by \cite[Corollary 2.14]{HN}.
Since the complex is bounded, the associated
$t$-structure is an intermediate $t$-structure.
The result then follows from Theorem \ref{der eq
restrictable} and from the fact that a cosilting $t$-structure induces a derived equivalence if and
only if it is cotilting (\cite[Corollary 5.2]{PV}).
\end{proof}

Taking Theorem \ref{der eq restrictable} into account, the assumption that $\Tbb$ induces a derived equivalence in Corollary \ref{cor her ft in some hearts} is redundant, therefore leading to the following simplification. 

\begin{cor}\label{tp finite type restrictable}
Let $R$ be a commutative noetherian ring and $\mathbb{T}$ an intermediate restrictable compactly generated $t$-structure in $\D(R)$ with heart $\Hcal_\Tbb$. Then there is a bijection between hereditary torsion pairs of finite type in $\Hcal_\Tbb$ and specialisation-closed subsets of $\Spec(R)$.
\end{cor}

At this point, one might speculate whether every intermediate compactly generated $t$-structure
leads to a derived equivalence. We instead show an example of such a $t$-structure that
does not induce a derived equivalence. In other words, this provides (implicitly) an example of a
bounded (3-term) pure-injective cosilting complex which is not cotilting over a commutative
noetherian ring. Note that by Corollary \ref{2-term cosilting is cotilting} such an example cannot
be found among 2-term cosilting complexes.

\begin{ex}\label{ex:Z-secondtilt}
Recall the situation considered in Example \ref{ex:Z} and assume, furthermore, that $R$ is
connected, i.e. that it has no non-trivial idempotent elements. With the same notation,
$\Hcal_{\mathbf{t}}$ is the heart of the $t$-structure corresponding to the sp-filtration
$\Spec{R}\supseteq V\supseteq\emptyset$. By Lemma~\ref{generalisation-closed support} we know that the set $V$ also corresponds to a hereditary torsion pair (of finite type, by Theorem \ref{classification of tp in heart}) in $\Hcal_\mathbf{t}$, namely $\mathbf{s}=(\Tcal,\Ccal[1])$, where $\Ccal$ is the Giraud subcategory associated to $\Tcal$ in $\mathsf{Mod}(R)$. Consider the heart $\Hcal_\mathbf{s}$ of the HRS-tilt of the $t$-structure with heart $\Hcal_{\mathbf{t}}$ with respect to $\mathbf{s}$. The corresponding $t$-structure, by Lemma~\ref{tilt between cg} is associated to the intermediate sp-filtration $\Spec(R)\supseteq V\supseteq V\supseteq \emptyset$. Notice that this filtration does not satisfy the weak Cousin condition and, hence, this $t$-structure is not restrictable.

By construction, we have $\Hcal_\mathbf{s}=\Ccal[2]\ast\Tcal$. Notice that since
$\mathbf{t}$ is perfect, for all objects $T$ in $\Tcal$ and $C$ in $\Ccal$ we have
$\Hom_{\D(R)}(T,C[3])\simeq\Ext_R^3(T,C)=0$ and hence all triangles
\[C[2]\longrightarrow X\longrightarrow T\longrightarrow C[3]\]
split. In other words, the torsion pair $(\Ccal[2],\Tcal)$ in $\Hcal_\mathbf{s}$ is a split torsion
pair. Moreover, the same argument shows that $\Hom_{\D(R)}(T,C[2])\simeq\Ext_R^2(T,C)=0$ and, thus, we
have that in fact also $(\Tcal,\Ccal[2])$ is a torsion pair in $\Hcal_\mathbf{s}$. In other words,
$\Ccal[2]$ and $\Tcal$ are abelian subcategories of $\Hcal_\mathbf{s}$ and $\Hcal_\mathbf{s}\cong
\Ccal[2]\times \Tcal$.
	
Now, since $R$ is connected, it follows that $\D(R)$ is an indecomposable triangulated category, i.e. it is not the product of two triangulated subcategories (see \cite[Example 3.2]{Bri0}). However, it is clear that $\D(\Hcal_\mathbf{s})$ is not indecomposable, as it is equivalent to the product $\D(\Ccal[2])\times \D(\Tcal)$. As a consequence, $\Hcal_\mathbf{s}$ cannot be derived equivalent to $\Mod(R)$. Note that, in particular, this provides an example of a cosilting (3-term) object of $\D(R)$ which is not cotilting.
\end{ex}

We conclude the paper exploring some consequences for the hearts of $t$-structures of
$\D^b(\mod(R))$.

\begin{prop}
Let $R$ be a commutative noetherian ring, and $\mathbb{B}$ a bounded $t$-structure of
$\D^b(\mod(R))$, with heart $\Bcal$. Then $\Bcal$ is the category of finitely presented
objects of a locally coherent Grothendieck category which is derived equivalent to $\Mod(R)$. Moreover, Serre subcategories of $\Bcal$ are in bijection with specialisation-closed subsets of $\Spec(R)$.
\end{prop}
\begin{proof}
Consider the compactly generated $t$-structure $\Theta(\mathbb{B})$. It is intermediate because so
is $\mathbb{B}$, and it is restrictable by construction. Hence, its heart $\Hcal_\Tbb$ is derived
equivalent to $\Mod(R)$, by Theorem \ref{der eq restrictable}. Now, by Theorem \ref{MZ}, $\Hcal_\Tbb$ is a 
locally coherent Grothendieck category with $\fp(\Hcal_\Tbb)=\Hcal_\Tbb\cap\D^b(\mod(R))=\Bcal$.
Finally, since Serre subcategories of $\Bcal$ are in bijection with hereditary
torsion pairs of $\Hcal_\Tbb$ (see \cite{H,Kr}), and therefore with specialisation closed subsets of
$\Spec(R)$ by Corollary \ref{tp finite type restrictable}.
\end{proof}

\end{document}